\documentclass[11pt]{amsart}
\usepackage{geometry}                
\usepackage{graphicx}
\usepackage{amssymb}
\usepackage{epstopdf}
\usepackage{hyperref}

\usepackage{color,ulem,textcomp}

\def\eps{\varepsilon}
\def\R{{\mathbb R}}
\def\N{{\mathbb N}}

\def\C{{\mathbb C}}
\def\R{{\mathbb R}}
\def\N{{\mathbb N}}
\def\Sch{{\mathcal S}}

\def\U{\mathcal U}
\def\wp{\mathcal{WP}}

\def\op{{\rm op}}

\def\eps{\varepsilon}

\def\Im{\text{\rm Im}}

\def\e{{\rm e}}

\def\<{{\langle}}
\def\>{{\rangle}}
\def\({{\left(}}

\def\beq{\begin{equation}}   \def\eeq{\end{equation}}
\def\bea{\begin{eqnarray}}  \def\eea{\end{eqnarray}}

\newcommand{\1}{\mathbb  I}

\theoremstyle{plain}
\newtheorem{theorem}{Theorem}[section]
\newtheorem{definition}[theorem]{Definition}
\newtheorem{assumption}[theorem]{Assumption}
\newtheorem{lemma}[theorem]{Lemma}
\newtheorem{corollary}[theorem]{Corollary}
\newtheorem{proposition}[theorem]{Proposition}

\theoremstyle{remark}
\newtheorem{remark}[theorem]{Remark}

\newtheorem{example}[theorem]{Example}

\def\ech{\color{black}\,}

\begin{document}

\title[]{Propagation of wave packets for systems presenting codimension one crossings}
\author[C. Fermanian Kammerer]{Clotilde~Fermanian-Kammerer}
\address{LAMA, UMR CNRS 8050,
Universit\'e Paris EST,
61, avenue du G\'en\'eral de Gaulle, 
94010 Cr\'eteil Cedex\\ France}
\email{Clotilde.Fermanian@u-pec.fr}
\author[C. Lasser]{Caroline Lasser}
\address{Zentrum Mathematik - M8
Technische Universit\"at M\"unchen
85747 Garching bei M\"unchen, Germany}
\email{classer@ma.tum.de}

\author[D. Robert]{Didier Robert}
\address{Laboratoire de math\'ematiques Jean Leray
UMR 6629 du CNRS
Universit\'e de Nantes
2, rue de la Houssini\`ere  44322 Nantes Cedex 3, France}
\email{didier.robert@univ-nantes.fr}

\begin{abstract} 
We analyze the propagation of wave packets through general Hamiltonian systems presenting 
codimension one eigenvalue crossings. The class of time-dependent Hamiltonians we consider is of general 
pseudodifferential form with subquadratic growth. It comprises Schr\"odinger operators with matrix-valued potential, 
as they occur in quantum molecular dynamics, but also covers matrix-valued models of solid state physics 
describing the motion of electrons in a crystal. We calculate precisely the non-adiabatic effects 
of the crossing in terms of a transition operator, whose action on coherent states can be spelled out 
explicitly.
\end{abstract}


\maketitle


\section{Introduction}

We consider systems of $N\ge2$ equations of pseudodifferential form 
\begin{equation}\label{system}
i\eps\partial_t \psi^\eps = \widehat H(t) \psi^\eps,\;\;\psi^\eps_{|t=t_0}=\psi^\eps_0,
\end{equation}
where $(\psi^\eps_0)_{\eps>0}$ is a bounded family in $L^2(\R^d,\C^N)$. 
The Hamiltonian operator 
\[
\widehat H(t) = H(t,x,-i\eps\nabla_x)
\] 
is the semi-classical Weyl quantization of a time-dependent Hamiltonian
\[
H: \R\times\R^{d}\times\R^d \to \C^{N\times N}, \quad (t,x,\xi)\mapsto  H(t,x,\xi),
\]
that is a smooth matrix-valued function and satisfies suitable growth conditions guaranteeing a well-defined 
and unique solution of the system. We denote with a ``$\widehat \cdot$'' the semi-classical Weyl quantization, the definition of which is recalled in Section~\ref{sec:functionspaces}. Phase space variables are denoted by $z=(x,\xi)\in \R^{2d}$. The semi-classical parameter $\eps>0$ is assumed to be small.
The initial data are wave packets associated with one of the eigenspaces 
of the Hamiltonian matrix. That is, 
\begin{equation}\label{initialdata}
\psi^\eps_0 = \widehat {\vec V_0}\,{\mathcal{WP}}^\eps_{z_0}\varphi_0,
\end{equation}
where $\vec V_0(z)$ is a normalized eigenvector of the matrix $H(t_0,z)$ such that $\vec V_0:\R^{2d}\to\C^N$ is 
a smooth vector-valued function, and ${\mathcal{WP}}^\eps_{z_0}\varphi_0$ denotes the wave packet 
transform of a Schwartz function $\varphi_0\in{\mathcal S}(\R^d,\C)$ for a phase space point 
$z_0=(x_0,\xi_0)\in\R^{2d}$,
 \beq\label{wpdef}
{\mathcal{WP}}^{\eps}_{z_0}\varphi_0(x)= \eps^{-d/4} \,{\rm e}^{i\xi_0\cdot(x-x_0)/\eps} 
\varphi_0\!\left(\tfrac{x-x_0}{\sqrt\eps}\right).
\eeq

\medskip
Our aim is to describe the structure of the system's solutions in the case, when eigenvalues of the 
Hamiltonian matrix $H(t,z)$ coincide for some point $(t,z)\in\R\times\R^{2d}$, while all eigenvalues and 
eigenvectors retain their smoothness. The literature refers to them as 
codimension one crossings. In the presence of eigenvalue crossings the key assumption for 
space-adiabatic theory (the existence of a positive gap between eigenvalues) is violated, and the knowledge of the dynamics associated with one of the 
eigenvalues is not enough any more. Moreover, in addition to the necessity to include more than one 
eigenvalue for an effective dynamical description, also the non-adiabatic transitions between 
the coupled eigenspaces have to be properly resolved. 
These questions have already been addressed for special systems corresponding to the following
physical settings: In his monograph \cite[Chapter~5]{Hag94}, G.~Hagedorn investigated 
Schr\"odinger Hamiltonians with matrix-valued potential,
\begin{equation}\label{ex:hag}
\widehat H_S = -\frac{\eps^2}{2}\Delta_x\, \1_{\C^N} + V(x),\quad V\in
{\mathcal C}^\infty(\R^d,\C^{N\times N}). 
\end{equation}
More recently, in \cite{WW}, A. Watson and M. Weinstein studied models arising in solid state 
physics in the context of Bloch band decompositions, 
\begin{equation}\label{ex:WW}
\widehat H _A= A(-i\eps\nabla_x) + W (x) \1_{\C^2},\quad A\in{\mathcal C}^\infty(\R^d,\C^{N\times N}),\quad 
W\in{\mathcal C}^\infty(\R^d,\C). 
\end{equation}
In both settings, the eigenvalues of the matrices $V(x)$, $x\in\R^d$, respectively $A(\xi)$, $\xi\in\R^d$, have a codimension one crossing of their eigenvalues.  

\medskip
We  develop here a new analytical method which applies for  general matrix-valued Hamiltonians
with a codimension one crossings of eigenvalues, which might also have multiplicity larger than one. In particular, we give a general and unified computation of the transfer operator which describes the non-adiabatic interactions due to the crossing. The non-adiabatic transition formulae of Corollary~\ref{cor:gaussian}
are explicit and derived in a self-contained and more accessible way than the previous ones in the literature. Due to their explicit form, they can directly be applied to numerical simulations based on thawed Gaussians that are currently investigated in chemical physics, see for example \cite{Vanicek} or the recent review 
\cite{VB}. 

\medskip
As another byproduct of our method, we also obtain an effortless generalization of the semi-classical Herman--Kluk approximation to the case of systems with eigenvalue gaps (see Corollary~\ref{cor:hk} below). We expect that 
a refinement of the present error analysis is possible, such that our codimension one result can be extended to the Herman--Kluk framework as well. This is work in progress, which might also contribute to the algorithmical development of superpositions of surface-hopping approximations using frozen (or thawed) Gaussian wave packets in the spirit of \cite{WH}, 
see also \cite{Lu}.

\medskip 
We assume that the matrix $H(t,z)$ has a smooth eigenvalue $h_1(t,z)$, the eigenspace of which admits a smooth eigenprojector $\Pi_1(t,z)$, that is,
\[
H(t,z) \Pi_1(t,z) = \Pi_1(t,z) H(t,z) = h_1(t,z) \Pi_1(t,z). 
\] 
We shall consider two situations, depending on whether the  eigenvalue $h_1(t,z)$ crosses another smooth  eigenvalue $h_2(t,z)$ or not. Because we assume  the Hamiltonian matrix $H(t,z)$ to be independent of $\eps$, then, in {\it the  gap situation}, the eigenvalue $h_1(t,z)$ is separated from $h_2(t,z)$ by a gap larger than some fixed positive real number $\delta_0>0$ that is  of order one with respect to the semi-classical parameter $\eps$. 
In the second case, {\it the smooth crossing} case,  both eigenvalues
 are smooth and  have  smooth eigenprojectors. Note that  it is not  the case in general since eigenvalues may develop singularities at the crossing; however,  we do not consider that situations here. 
We shall also assume that $H(t,z)$ has no other eigenvalues since one can reduce to that case as soon as  the set of these two eigenvalues is separated from the remainder of the spectrum of the matrix~$H(t,z)$ by a gap (uniformly in $t$ and $z$). 

\medskip 
The gap situation is well understood  and corresponds to adiabatic situations that have been studied by several authors (see in particular the lecture notes~\cite{Te} of S. Teufel or the memoirs\cite{MS} of A. Martinez and V. Sordoni and note that the thesis~\cite{bi} is devoted to wave packets in the adiabatic situation). For avoided crossings, the coupling of the gap and the semi-classical parameter violates the key requirement for adiabatic decoupling. 
The resulting non-adiabatic dynamics have been studied for wave packets by G. Hagedorn and A. Joye in 
\cite{HagJ98,HagJ99} and for the Wigner function of general initial data in \cite{FL}. 
Smooth crossings have been less studied so far. Some results on the subject focus on the evolution at leading order in $\eps$ of quadratic quantities of the wave function for initial data which are not necessarily wave packets (see~\cite{J,DFJ} and the references therein).  The main results devoted to wave packet propagation through smooth eigenvalue crossings are the references~\cite{Hag94} and~\cite{WW} mentioned above. There, for the specific Hamiltonian operators \eqref{ex:hag} and \eqref{ex:WW}, respectively, the authors gave rather explicit descriptions of the propagated wave packet,  exhibiting non-adiabatic transitions that occur at the crossing between the two eigenvalues that are of order $\sqrt\eps$. As in these contributions, we assume  that the crossing set
\begin{equation}\label{def:Upsilon}
\Upsilon=\{(t,z)\in\R^{2d+1},\; h_1(t,z)=h_2(t,z)\}
\end{equation}
of two smooth eigenvalues $h_1(t,z)$ and $h_2(t,z)$ is a codimension one manifold.

\medskip
Our main result (Theorem~\ref{theo:WPcodim1} below) makes the following assumptions for the initial data~$\psi^\eps_0$. 
Let $v^\eps_0$ be a wave packet centered in a phase space point $z_0$, that is, 
\[
v^\eps_0 = {\mathcal{WP}}_{z_0}^\eps \varphi_0\quad\text{for some}\quad \varphi_0\in{\mathcal S}(\R^d,\C).
\]
Let $\vec V_0(z)$ be a smooth normalized eigenvector of $H(t_0,z)$, that is, $\vec V_0\in{\mathcal C}^\infty(\R^{2d},\C^N) $ is a smooth vector-valued function that satisfies in a neighborhood $U$ of $z_0$,
\[
H(t_0,z)\vec V_0(z)=h_1(t_0,z)\vec V_0(z)\quad\text{for all}\quad z\in U.
\]
Then, we define the initial wave packet according to~\eqref{initialdata}.
Let $z_1(t)$ denote the classical trajectory associated with the eigenvalue $h_1(t,z)$ initiated in wave packet's core $z_0$. Let $t^\flat>t_0$ be the first time, when the trajectory $z_1(t)$ meets the crossing set $\Upsilon$, 
and let $z_2(t)$ denote the classical trajectory associated with the second eigenvalue $h_2(t,z)$, 
that is initiated in the crossing point $z_1(t^\flat)$. That is, 
\begin{align*}
\dot z_1(t) &= J \partial_z h_1(t,z_1(t)),\quad z_1(t_0)=z_0,\\
\dot z_2(t) &= J \partial_z h_2(t,z_2(t)),\quad z_2(t^\flat)=z_1(t^\flat).
\end{align*}\ech
Then, the solution of system \eqref{system} satisfies 
$$\psi^\eps(t) = \widehat { \vec V_1(t)} {\mathcal{WP}}^\eps_{z_1(t)} (\varphi_1^0(t)+\sqrt\eps\varphi_1^1(t)) + \sqrt\eps {\bf 1}_{t>t^\flat} \widehat { \vec V_2(t)} {\mathcal{WP}}^\eps_{z_2(t)} \varphi_2(t) +o(\sqrt\eps),$$
where the profiles of the wave packets 
\[
{\mathcal{WP}}^\eps_{z_1(t)} (\varphi_1^0(t)+\sqrt\eps\varphi_1^1(t))\quad\text{and}\quad
{\mathcal{WP}}^\eps_{z_2(t)} \varphi_2(t)\] 
are Schwartz functions 
$\varphi_1^0(t)$, $\varphi_1^1(t)$, and $\varphi_2(t)$,  
that  solve $\eps$-independent PDEs on $[t_0,t^\flat]$ and $[t^\flat,t_0+T]$, respectively, that are explicitly given in terms of the classical dynamics associated with the eigenvalues $h_1(t,z)$ and $h_2(t,z)$. 
The profile associated with the second eigenvalue is generated by the leading order profile of the first eigenvalue via
$$ \varphi_2(t^\flat) ={\mathcal T}^\flat  \varphi_1^0(t^\flat),$$
where the non-adiabatic transfer operator ${\mathcal T}^\flat$ is a metaplectic transform (which implies that the structure of Gaussian states is preserved, see Corollary~\ref{cor:gaussian}). The two families $\vec V_1(t,z)$ and $\vec V_2(t,z)$ are smooth normalized eigenvectors for $h_1(t,z)$ and $h_2(t,z)$, respectively, that are obtained by  parallel transport.

\medskip 
We point out that, in the uniform gap case, an initial datum that is associated with one eigenvalue issues 
a solution at time $t$ that is associated with the same eigenvalue up to terms of order $\eps$, which is the standard order of the adiabatic approximation, while for smooth crossings a perturbative term of order $\sqrt\eps$ associated with the other eigenvalue has to be taken into account for an order $\eps$ approximation.

\medskip 
Before giving a more precise statement of the result, we mention that the propagation of  wave packets was also studied for  nonlinear systems in~\cite{CF11,Hari1,Hari2}, including situations with avoided crossings~\cite{Hari2}. However, nonlinear systems with codimension one crossings have not yet been analysed. We expect that our result can be extended when imposing appropriate assumptions on the nonlinearity. 

\medskip
 \noindent{\bf Acknowledgements}. Didier Robert thanks Jim Ralston for his comments on a first version of our paper, Clotilde Fermanian Kammerer thanks the Von Neumann Professorship program of the Technische Universit\"at M\"unchen which gives her the opportunity to work on this article during the academic year 2019, and Caroline Lasser thanks the I-Site Future program for the visiting professorship 2020. Part of this work has also been 
 supported by the CNRS 80$\,|\,$Prime project.

\section{Preliminary results}
In this section, we introduce the relevant function spaces for the unitary propagation and also recall some known results on wave packets for scalar evolution equations. 

\subsection{Function spaces and quantization}\label{sec:functionspaces}
Let $a\in{\mathcal C}^\infty(\R^{2d})$ be a smooth scalar-, vector- or matrix-valued function with adequate control on the growth of derivatives. Then, the Weyl operator $\widehat a = \op^w_\eps(a)$ is defined by 
$$\op^w_\eps(a)f(x):= \widehat a f(x) := (2\pi\eps)^{-d} \int_{\R^{2d}} a\!\left({x+y\over 2}, \xi\right) 
{\rm e}^{i\xi\cdot(x-y)/\eps} f(y) \,dy\, d\xi$$
for all $f\in{\mathcal S}(\R^d)$.  
According to \cite{MaRo}, the unitary propagator ${\mathcal U}^\eps_H(t,t_0)$ associated with the Hamiltonian 
operator $\widehat H(t)$,
\[
i\eps\,\partial_t\, {\mathcal U}^\eps_H(t,t_0) = \widehat H(t)\, {\mathcal U}^\eps_H(t,t_0),\quad 
{\mathcal U}^\eps_H(t_0,t_0) = \1_{L^2(\R^d)},
\]
is well defined  when the map $(t,z)\mapsto H(t,z)$ is in ${\mathcal
  C}^\infty(\R\times \R^{2d},\C^{N\times N})$, valued in the set of self-adjoint matrices  and that it has  subquadratic growth, i.e.
\begin{equation}\label{hyp:H}
\forall \alpha\in \N^{2d} ,\;\;|\alpha|\geq 2,\;\;
\exists C_\alpha>0,\;\;\sup_{(t,z)\in\R\times \R^{2d}}\| \partial^\alpha_z H(t,z) \|_{\C^{N\times N}}\leq C_\alpha.
\end{equation}
These assumptions  guarantee the existence of solutions to equation~(\ref{system}) in $L^2(\R^d,\C^N)$ and, more generally, in the  functional spaces 
$$\Sigma_\eps^k(\R^d)=\{ f\in L^2(\R^d),\;\;\forall \alpha,\beta\in\N^d,\;\; |\alpha|+|\beta| \leq k,\;\; x^\alpha (\eps \partial_x)^\beta f\in L^2(\R^d)\}$$
endowed with the norm 
$$\| f\|_{\Sigma^k_\eps} = \sup_{|\alpha|+|\beta| \leq k}\| x^\alpha (\eps \partial_x)^\beta f\|_{L^2}.$$
We note that also with respect to the $\Sigma_\eps^k(\R^d)$ spaces, the unitary propagator ${\mathcal U}^\eps_H(t,t_0)$ is $\eps$-uniformly-bounded in the sense, that for all $T>0$ there exists $C>0$ such that 
\[
\sup_{t\in[t_0,t_0+T]}\|{\mathcal U}^\eps_H(t,t_0)\|_{{\mathcal L}(\Sigma^k_\eps)} \,\le\, C.
\]

\begin{remark}
The analysis below could apply to more general settings as long as the classical quantities are well-defined in finite time with some technical improvements that are not discussed here.
\end{remark} 
 
\subsection{Scalar propagation and scalar classical quantities}\label{sec:defclas}

 The most interesting property of the coherent states is the stability of their structure through evolution, which can be described by means of classical quantities. Note that  for all $z\in\R^{2d}$ and $k\in\N$, the operator $\varphi\mapsto \wp^\eps_{z}\varphi$
is a unitary map in~$L^2(\R^d)$ which maps continuously  $\Sigma^1_k$  into $\Sigma^\eps_k$ 
with a continuous inverse. Other elementary properties of the wave packet transform are listed in 
Lemma~\ref{lem:coherent}.
We shall use the notation
\begin{equation}\label{def:J}
J = \begin{pmatrix}0 & \1_{\R^d}\\ -\1_{\R^d} & 0\end{pmatrix}.
\end{equation}
For smooth functions 
$f,g\in{\mathcal C}^\infty(\R^{2d})$, that might be scalar-, vector- or matrix-valued, we denote  the Poisson bracket by
$$
\{f,g\}:= J\nabla f\cdot \nabla g= \sum_{j=1}^d \left(\partial_{\xi_j} f \partial_{x_j} g- \partial_{x_j} f \partial_{\xi_j} g\right).
$$
Let $h:\R\times\R^{2d}\to\R$, $(t,z)\mapsto h(t,z)$ be a smooth function of subquadratic growth\eqref{hyp:H}.We now review the main tools for the semi-classical propagation of wave-packets. We let $z(t) = (q(t),p(t))$ denote the {\it classical Hamiltonian trajectory} issued from a phase space point $z_0$ at time $t_0$, that is defined by the ordinary differential equation
$$\dot z(t) = J \partial_z h(t,z(t)),\;\; z(t_0)=z_0.$$
The trajectory $z(t) = z(t,t_0,z_0)$ depends on the initial datum and defines 
via $\Phi_h^{t,t_0}(z_0) = z(t,t_0,z_0)$ the associated {\it flow map} of the Hamiltonian function $h$. 
We will also use the trajectory's {\it action integral}
\begin{equation}
\label{def:S} 
S(t,t_0,z_0) = \int_{t_0}^t \left(p(s)\cdot \dot q(s)-h(s,z(s)) \right) ds,
\end{equation}
and the Jacobian matrix of the flow map 
\[
F(t,t_0,z_0) = \partial_z \Phi_h^{t,t_0}(z_0). 
\]
Note that $F(t,t_0,z_0)$ is a symplectic $2d\times 2d$ matrix, that satisfies the linearized flow equation
\begin{equation}\label{eq:lin}
\partial_t F(t,t_0,z_0) = J {\rm Hess}_zh(t,z(t)) \, F(t,t_0,z_0),\;\;F(t_0,t_0,z_0) = \1_{\R^{2d}}.
\end{equation}
We denote its blocks by 
\begin{equation}\label{eq:F}
F(t,t_0,z_0) = 
 \begin{pmatrix}  A(t,t_0,z_0) &B(t,t_0,z_0) \\ C(t,t_0,z_0) &D(t,t_0,z_0)\end{pmatrix}.
\end{equation}
In a last step, we define the corresponding unitary evolution operator, the {\it metaplectic transformation}, 
that acts on square integrable functions in $L^2(\R^d)$.

\begin{definition}[Metaplectic transformation]
Let $h:\R\times\R^{2d}\to\R$ be a smooth function of subquadratic growth~(\ref{hyp:H}). Let $t,t_0\in\R$ and $z_0\in\R^{2d}$. 
Let $F(t,t_0,z_0)$ be the solution of the linearized flow equation \eqref{eq:lin} associated with the Hamiltonian function $h(t)$. Then, we call the unitary operator 
 \[
  {\mathcal M}[F(t,t_0,z_0)] : \; \varphi_0\mapsto \varphi(t) 
\]
that associates with an initial datum $\varphi_0$ the solution at time $t$ of the Cauchy problem 
$$
i\partial_t \varphi= \op^w_1({\rm Hess}_z h(t, z(t))z\cdot z) \varphi,\;\; \varphi(t_0)=\varphi_0,
$$
the {\rm metaplectic transformation} associated with the matrix $F(t,t_0,z_0)$. 
\end{definition}

Using these three $\eps$-independent building blocks -- the classical trajectories, the action integrals, and the 
metaplectic transformations associated with the linearized flow map -- we can approximate 
the action of the unitary propagator  
\[
i\eps\partial_t\, {\mathcal U}_h^\eps(t,t_0) = \op^w_\eps(h(t))\, {\mathcal U}_h^\eps(t,t_0),\quad 
{\mathcal U}_h^\eps(t_0,t_0) = \1_{L^2(\R^d)}
\]
on wave packets as follows.\ech

\begin{proposition}\label{wpevol1}[\cite[\S4.3]{CombescureRobert}] 
Consider a smooth scalar Hamiltonian $h(t)$ of subquadratic growth~(\ref{hyp:H}). 
Let $T>0$, $k\ge 0$, $z_0\in\R^{2d}$, and $\varphi_0\in{\mathcal S}(\R^d)$. 
Then, there exists a positive constant $C>0$ such that
\[
\sup_{t\in[t_0,t_0+T]}\left\|{\mathcal U}_h^\eps(t,t_0) \wp^\eps_{z_0}\varphi_0 -  
{\rm e}^{\frac{i}{\eps}S(t,t_0,z_0)} \wp^\eps_{z(t)} \varphi^\eps(t)\right\|_{\Sigma_\eps^k} \le C\eps,
\]
where the profile function $\varphi^\eps(t)$ is given by
\begin{equation}\label{eq:profile}
\varphi^\eps(t) = {\mathcal M}[F(t,t_0,z_0)] \left(1+\sqrt\eps\, b_1(t,t_0,z_0)\right) \varphi_0,
\end{equation}
and the correction function $b_1(t,t_0,z_0)$ satisfies
\begin{equation}\label{rodrigues}
b_1(t,t_0,z_0) \varphi_0 = 
\sum_{\vert\alpha\vert=3} \frac{1}{\alpha!} \frac{1}{i} \int_{t_0}^{t} \partial_z^\alpha h(s,z(s))\, 
\op_1^w[(F(s,t_0,z_0)z)^{\alpha}] \,\varphi_0\,ds. 
\end{equation}
The constant $C=C(T,k,z_0,\varphi_0)>0$ is independent of $\eps$ but depends on derivative bounds 
of the flow map $\Phi^{t,t_0}_h(z_0)$ for $t\in[t_0,t_0+T]$ and the $\Sigma^{k+3}_1$-norm of the 
initial profile $\varphi_0$.
\end{proposition}

Let us discuss the especially interesting case of initial Gaussian states. 
{\it Gaussian states} are wave packets with complex-valued Gaussian profiles, whose  
covariance matrix is taken in the Siegel half-space ${\mathfrak S}^ +(d)$ of  $d\times d$ complex-valued symmetric matrices with positive imaginary part,
\[
{\mathfrak S}^+(d) = \left\{\Gamma\in\C^{d\times d},\  \Gamma=\Gamma^\tau,\ \Im\Gamma >0\right\}.
\]
With  $\Gamma\in{\mathfrak S}^+(d)$ we associate the Gaussian profile
\begin{equation}\label{def:gaussian}
g^\Gamma(x)
 := c_\Gamma\, {\rm e}^{\frac{i}{2}\Gamma x\cdot x},\quad x\in\R^d,
\end{equation}
where 
$c_\Gamma=\pi^{-d/4} {\rm det}^{1/4}(\Im\Gamma)$  
is a normalization constant in $L^2(\R^d)$. It is a non-zero complex number whose argument is determined by
 continuity according to the working environment.
By Proposition~\ref{wpevol1},  the  Gaussian states remain Gaussian under the evolution by 
${\mathcal U}^\eps_h(t,t_0)$. Indeed, for  $\Gamma_0\in{\mathfrak S}^+(d)$, we have 
\[
 {\mathcal M}[F(t,t_0,z_0)] g^{\Gamma_0}= g^{\Gamma(t,t_0,z_0)},
\]
where the width $\Gamma(t,t_0,z_0)\in\mathfrak S^+(d)$ and the corresponding normalization $c_{\Gamma(t,t_0,z_0)}$ are determined by the initial width $\Gamma_0$ and the Jacobian $F(t,t_0,z_0)$ according to
\begin{eqnarray}\label{def:Gamma}
\Gamma(t,t_0,z_0) 
&=& (C(t,t_0,z_0)+ D(t,t_0,z_0)\Gamma_0)(A(t,t_0,z_0) +B(t,t_0,z_0)\Gamma_0)^{-1}\\
\nonumber
c_{\Gamma(t,t_0,z_0)} 
&=& c_{\Gamma_0}\,{\rm det}^{-1/2}(A(t,t_0,z_0)+B(t,t_0,z_0)\Gamma_0).
\end{eqnarray}
The branch of the square root in ${\rm det}^{-1/2}$ is determined by continuity in time.

\medskip
The semiclassical wave packets used by G. Hagedorn in \cite{Hagedorn, Hagedorn81} are Gaussian wave packets, which are multiplied with a specifically chosen complex-valued polynomial function, that depends on the Gaussian's width matrix. If $A\in \mathcal C^\infty (\R^{2d},\C)$ is an arbitrary polynomial function, then ${\rm op}^w_1(A)  g^{\Gamma_0}$ is the product of a polynomial times a Gaussian, and we can again describe the 
action of the metaplectic transformation explictly. Indeed, by Egorov's theorem (which is exact here), 
\begin{align*}
{\mathcal M}[F(t,t_0,z_0)]({\rm op}_1^w(A)g^{\Gamma_0}) &= {\rm op}_1^w(A\circ F(t,t_0,z_0)) {\mathcal M}[F(t,t_0,z_0)] g^{\Gamma_0} \\
&={\rm op}_1^w(A\circ F(t,t_0,z_0)) g^{\Gamma(t,t_0,z_0)}.
\end{align*}
In particular, functions that are polynomials times a Gaussian remain of the same form under the evolution, even the polynomial 
degree is preserved.

\section{Precise statement of the results}
We now present our main results, that extend the previous theory of wave packet propagation for scalar evolution 
equations to systems associated with Hamiltonians that have smooth eigenvalues and eigenprojectors.

\subsection{Vector-valued wave packets and parallel transport}
We consider initial data that are vector-valued wave packets associated with a normalized  
eigenvector of the Hamiltonian matrix $H(t_0,z)$ as given in \eqref{initialdata}. 
The evolution of such a function also requires an appropriate evolution of its vector part, 
which we refer to as {\it parallel transport}.
The following construction generalizes \cite[Proposition~1.9]{CF11}, which was inspired by the work of
G.~Hagedorn, see \cite[Proposition~3.1]{Hag94}. Let us denote the complementary orthogonal projector 
by 
$\Pi^\perp(t,z) = \1_{\C^N}-\Pi(t,z)$ 
and assume that 
\begin{equation}\label{eq:decompose}
H(t,z) = h(t,z)\Pi(t,z) + h^\perp(t,z)\Pi^\perp(t,z)
\end{equation}
with the second eigenvalue given by 
$h^\perp(t,z) = {\rm tr}(H(t,z)) - h(t,z).$
We introduce the auxiliary matrices 
\begin{align}\label{def:Omega}
\Omega(t,z) &=-\tfrac12\big(h(t,z)-h^\perp(t,z)\big)\Pi(t,z)\{\Pi,\Pi\}(t,z)\Pi(t,z) ,\\*[1ex]
\label{def:K}
K(t,z) &= \Pi^\perp(t,z)\left(\partial_t\Pi(t,z)+\{h,\Pi\}(t,z)\right)\Pi(t,z),\\*[1ex]
\label{def:Theta} 
\Theta(t,z) &= i\Omega(t,z) + i(K-K^*)(t,z),
\end{align}
that are smooth and satisfy algebraic properties detailed in Lemma~\ref{lem:B} below. In particular, 
$\Omega$ is skew-symmetric and $\Theta$ is self-adjoint,
$\Omega = -\Omega^*\quad$ and $\Theta = \Theta^*$.
 We note, that for the Schr\"odinger and the Bloch Hamiltonian,  
\[
H_S(z) = \tfrac12|\xi|^2\,\1_{\C^N} + V(x)\quad \text{and}\quad H_A(z) = \begin{pmatrix} 0 & \xi_1+i\xi_2 \\ \xi_1-i\xi_2 & 0\end{pmatrix}  + W(x)\1_{\C^2},
\] 
the skew-symmetric $\Omega$-matrix vanishes, that is, $\Omega_S = 0$ and $\Omega_A = 0$. 
For Dirac Hamiltonians with electromagnetic potential or Hamiltonians that describe 
acoustic waves in elastic media, the $\Omega$-matrix need not vanish.

\begin{proposition} \label{prop:eigenvector}
Let $H(t,z)$ be a smooth Hamiltonian with values 
in the set of self-adjoint $N\times N$ matrices that is of subquadratic growth~(\ref{hyp:H}) and has a smooth spectral decomposition~\eqref{eq:decompose}. We assume that both eigenvalues are of subquadratic growth as well. We consider $\vec V_0\in{\mathcal C}_0^\infty(\R^{2d},\C^N)$ and $z_0\in \R^{2d}$ such that there exits a neighborhood $U$ of $z_0$ such that for all $z\in U$
$$\vec V_0 (z)=\Pi(t_0,z)\vec V_0(z)\quad\text{and}\quad \|\vec V_0(z)\|_{\C^N} = 1.$$
Then, 
  there exists a smooth normalized vector-valued function $\vec V(t,t_0)$ satisfying 
  $$\vec V(t,t_0,z)= \Pi(t,z)\vec V(t,t_0,z)\quad\text{for all}\quad z\in \Phi_h^{t,t_0}(U),$$
   such that for all $t\in\R$ and $z\in\Phi_h^{t,t_0}(U)$,
\begin{equation}\label{eq:eigenvector}
\partial_t \vec V(t,t_0,z) + \{h, \vec V\} (t,t_0,z) = -i\Theta(t,z)\vec V (t,t_0,z),\;\; \vec V(t_0,t_0,z) = \vec V_0(z).
\end{equation} 
\end{proposition}
\ech

Proposition~\ref{prop:eigenvector} is proved in Appendix~\ref{app:parallel}. 
Note that it does not require any gap condition for the eigenvalues. We will use it in the crossing situation, with smooth eigenvalues and eigenprojectors. 

\medskip
The parallel transport is enough to describe at leading order  the propagation of wave-packets associated with an eigenvalue $h(t,z)$ of the matrix $H(t,z)$, that is uniformly separated from the remainder of the spectrum in the sense that  there exists $\delta>0$ such that 
for all $(t,z)\in\R\times\R^{2d}$,
\begin{equation}\label{hyp:gap}
{\rm dist}\left(h(t,z),\sigma(H(t,z))\setminus\{h(t,z)\}\right) > \delta.
\end{equation}
Note that, this gap assumption implies the existence of a Cauchy contour ${\mathcal C}$ in the complex plane, 
such that its interior only contains the eigenvalue $h(t,z)$ and no other eigenvalues of $H(t,z)$. 
Then, 
one can write
 the  eigenprojector as
$\Pi(t,z)=-{1\over 2\pi i}\oint_{\mathcal C} (H(t,z)-\zeta)^{-1} d\zeta,$
which implies that the projector 
$\Pi(t,z)$ inherits the smoothness properties of the Hamiltonian $H(t,z)$ in the presence of an eigenvalue gap.  However, if the symbol $\Pi$ is of course of matrix norm~$1$, its derivatives may grow as $|z|$ goes to infinity and we shall make assumption below (see~\eqref{hyp:gapinfinity}) in order to guarantee that this growth is at most polynomial.
Since the pioneering work of T. Kato \cite{K1}, numerous studies have been devoted to this {\it adiabatic situation}  (see for example \cite{N1,N2,Te,MS} and references therein). One can derive from these results the following statement of adiabatic decoupling.

\begin{theorem}\label{thm:decoupling}[\cite{Te,MS,CF11}] Let $H(t,z)$ be a smooth Hamiltonian with values 
in the set of self-adjoint $N\times N$ matrices and $h(t,z)$ a smooth eigenvalue of $H(t,z)$. Assume that both $H(t,z)$ and $h(t,z)$ are of subquadratic growth~(\ref{hyp:H}) and that there exists an eigenvalue gap as in Assumption~\eqref{hyp:gap}. Consider initial data 
$(\psi^\eps_0)_{\eps>0}$ that are wave packets as in~(\ref{initialdata}).
Then, for all $T>0$, there exists $C>0$ such that $\psi^\eps(t) = \U^\eps_H(t,t_0)\psi^\eps_0$
satisfies the estimate
$$\sup_{t\in [t_0,t_0+T]} \left\| \widehat {\Pi^\perp(t)}\psi^\eps(t) \right\| _{L^2(\R^d)}  +  \left\| \psi^\eps(t) - \widehat {\vec V(t)} v^\eps(t) \right\| _{L^2(\R^d)}\leq C \eps $$
where $v^\eps(t)= {\mathcal U}^\eps_h(t,t_0) v^\eps_0$
and $\vec V(t)$ is determined by Proposition~\ref{prop:eigenvector}.
Besides, if there exists $k\in\N$ such that  $(\psi^\eps_0)_{\eps>0}$ is a bounded family in the space $\Sigma_\eps^k$, then the convergence above holds in $\Sigma^k_\eps$. 
\end{theorem}

Theorem~\ref{thm:decoupling}  is obtained as an intermediate result in the proof of Proposition~\ref{prop:propagationt**}, see Section~\ref{sec:outside} below. There, we perform a refined analysis 
of the adiabatic approximation that explicitly accounts for the size of the eigenvalue gap. We note 
that the estimate of Theorem~\ref{thm:decoupling} is unchanged, when allowing for perturbations 
of the initial data that are of order $\eps$ in $L^2(\R^d)$ or $\Sigma^k_\eps$, respectively. We also note, that in general the operator $\widehat \Pi(t)$ is {\it not} a projector, but coincides at order $\eps$ with the superadiabatic operators constructed in~\cite{MS,Te}, which are projectors (see also Appendix~\ref{App:projectors}).

\begin{remark}\label{rem:generalisation1}
The result of Theorem~\ref{thm:decoupling} can be generalized by means of superadiabatic projectors, showing that $\psi^\eps(t)$ can be approximated at any order by an asymptotic sum of wave packets.  The precise time evolution of coherent states was studied in the adiabatic setting in~\cite{bi,MS,Rcimpa}. These results are obtained via an asymptotic quantum diagonalization, in the spirit of the construction of the superadiabatic projectors of~\cite{MS,Te}. 
\end{remark}

Theorem~\ref{thm:decoupling} allows a semi-classical description of the dynamics of an initial wave packet, that is associated with a  gapped eigenvalue. The building blocks are the scalar classical quantities introduced in section~\ref{sec:defclas} and the parallel transport of eigenvectors given in Proposition~\ref{prop:eigenvector}. This is stated in the next Corollary; our aim is to derive a similar description for systems presenting a codimension one crossing. 

\begin{corollary}[Adiabatic wave packet]
In the situation of Theorem~\ref{thm:decoupling}, for any $T>0$, $k\in\N$, $z_0\in\R^{2d}$, and $\varphi_0\in{\mathcal S}(\R^d,\C)$, there exists a 
constant $C>0$ such 
$$
\sup_{t\in[t_0,t_0+T]}\left\|{\mathcal U}^\eps_{H}(t,t_0)\, \widehat {\vec V_0}\,  {\mathcal {WP}}^\eps_{z_0}\varphi_0 -  
{\rm e}^{i S(t,t_0,z_0)/\eps } \,
\widehat {\vec V(t,t_0)}\, {\mathcal{WP}}^\eps_{\Phi_{h}^{t,t_0}(z_0)}\varphi^\eps(t) \right\|_{\Sigma^k_\eps} \le C\eps,
$$
where the profile $\varphi^\eps(t)$ is given by~(\ref{eq:profile}), and all the classical quantities are associated with the eigenvalue $h(t)$. 
\end{corollary} 

We close this section devoted to gapped systems by formulating another semi-classical consequence of adiabatic theory using the {\it Herman--Kluk propagator}. This approximate propagator has first been proposed by M. Herman and  E. Kluk in \cite{HK} for scalar Schr\"odinger equations and later used as a numerical method for quantum dynamics in the semi-classical regime, see for 
example \cite{KluHD86} or more recently \cite{LS,Grossmann} with references therein. The rigorous 
mathematical analysis of the Herman--Kluk propagator is due to \cite{R,RS1}. The starting point of this 
approximation is the wave packet inversion formula 
$$
\psi(x)= (2\pi\eps)^{-d} \int_{z\in\R^{2d}}\langle g^\eps_z,\psi\rangle g^\eps_z (x) dz
$$
that allows to write any square integrable function $\psi\in L^2(\R^d)$ as a continuous superposition of 
Gaussian wave packets of unit width, 
\[
g^\eps_z(x) \ =\  \mathcal{WP}_z^\eps(g^{i\1})(x) \ =\  (\pi\eps)^{-d/4} 
\e^{-|x-q|^2/(2\eps) + i p\cdot(x-q)/\eps}.
\] 
The semi-classical description of unitary quantum dynamics within the framework of Gaussians of fixed unit width 
becomes possible due to a reweighting factor, the so-called {\it Herman--Kluk prefactor},
\[
a_h(t,t_0,z) = 
2^{-d/2}   \ {\rm det}^{1/2} \left( A(t,t_0,z)+D(t,t_0,z)+i(C(t,t_0,z)-B(t,t_0,z))     \right),
\]
which is solely determined by the blocks of the Jacobian matrix of the classical flow map. The resulting propagator 
\[
\psi\,\mapsto\, {\mathcal I}_h^\eps(t,t_0)\psi \,=\, (2\pi\eps)^{-d} \int_{\R^{2d}} 
\langle g^\eps_z,\psi\rangle  a_h(t,t_0,z) {\rm e}^{i S(t,t_0,z)/\eps} g^\eps_{\Phi^{t,t_0}_h(z)} dz
\]
provides an order $\eps$ approximation to the scalar unitary propagator ${\mathcal U}^\eps_h(t,t_0)$ in operator norm. Combining \cite[Proposition~2 and Theorem~2]{RS1} or \cite[Theorem~1.2]{R} with our previous results we obtain a Herman--Kluk approximation for gapped systems.
 
\begin{corollary}[Adiabatic Herman--Kluk approximation]   \label{cor:hk}
In the situation of Theorem~\ref{thm:decoupling}, for all $T>0$ there exists a constant $C=C(T)>0$ such that
$$
\sup_{t\in[t_0,t_0+T]}\left\| {\mathcal U}_H^\eps(t,t_0)\psi^\eps_0 - {\mathcal I}_H^\eps(t,t_0)\psi^\eps_0 
\right\|_{L^2(\R^d)} \le C\eps,
$$
where the vector-valued Herman--Kluk propagator is defined by 
\[
{\mathcal I}_H^\eps(t,t_0)\psi^\eps_0 = (2\pi\eps)^{-d} \int_{\R^{2d}} 
\langle g^\eps_z,v^\eps_0\rangle  \vec A(t,t_0,z) {\rm e}^{i S(t,t_0,z)/\eps} g^\eps_{\Phi^{t,t_0}_h(z)} dz.
\]
The prefactor $\vec A(t,t_0,z)$ is given by 
$\vec A(t,t_0,z) = \vec V(t,t_0,z) a_h(t,t_0,z)$,
where $a_h(t,t_0,z)$ is the Herman--Kluk prefactor associated with the eigenvalue $h(t)$. 
\end{corollary}

\medskip 
Theorem~\ref{thm:decoupling} formulates adiabatic decoupling for a single eigenvalue that is uniformly 
separated from the remainder of the spectrum. 
As it is well-known, adiabatic theory also extends to the situation where a subset  of eigenvalues is isolated from the remainder of the spectrum. For this reason, in the next section, we reduce our analysis to the  case of matrices with two eigenvalues that  coincide on a hypersurface~$\Upsilon$ of codimension one and differ away from it. We explicitly describe the dynamics of wave packets through this type of crossings, which is our main result.

\subsection{Main result: propagation of wave packets through codimension one crossings} 

We write the Hamiltonian matrix $H(t,z)$ as  
\begin{equation}\label{def:H22}
H(t,z)= v(t,z) \1_{\R^N} + H_0(t,z),\;\; v(t,z)=  {1\over N}{\rm tr} H(t,z),
\end{equation} 
where $v(t,z)$ is a real number and $H_0(t,z)$ a trace-free self-adjoint $N\times N$ matrix.
We assume that $H(t,z)$  has two smooth eigenvalues that cross on a hypersurface~$\Upsilon$. 
Such  a situation is called a {\it codimension one crossing} (see Hagedorn's classification~\cite{Hag94} for example). Let us formulate our assumptions on the crossing set more precisely.

\begin{assumption}[Codimension one crossing]\label{hyp:codim1} 
Let $H:\R^{2d+1}\to\C^{N\times N}$ be a smooth function with values in the set of self-adjoint $N\times N$ matrices that is of subquadratic growth~\eqref{hyp:H}. We assume:
\begin{enumerate}\setlength{\itemsep}{0.5ex}
\item[a)]
The matrix $H(t,z)$ has two smooth eigenvalues $h_1(t,z)$ and $h_2(t,z)$ that are of subquadratic growth~\eqref{hyp:H}.
\item[b)]
These eigenvalues cross on  a hypersurface~$\Upsilon$ of 
$\R^{2d+1}$ and differ outside of $\Upsilon$. In particular, for any $(t^\flat,z^\flat)\in\Upsilon$ there exists a neighbourhood~$\Omega$ of~$(t^\flat,z^\flat)$ and a smooth scalar function $(t,z)\mapsto f(t,z)$ defined on~$\Omega$ such that $f(t,z)=0$ is a local equation of~$\Upsilon$ with $d_{t,z}f\neq 0$ on~$\Omega$.
\item[c)]
The scalar function $(t,z)\mapsto v(t,z)$ defined by the decomposition \eqref{def:H22} satisfies  
\begin{equation}\label{hyp:transversality}
\partial_t f + \{v ,f\} \not=0\ \text{on}\ \Omega.
\end{equation}
\item[d)]
The crossing is non-degenerate in the sense, that the matrix $H_0(t,z)$ defined by the decomposition \eqref{def:H22} satisfies 
\[
H_0(t,z) = f(t,z) \tilde H_0(t,z)\ \text{on}\ \Omega
\] 
for some smooth matrix-valued map $(t,z) \mapsto \tilde H_0(t,z)$ with 
$\tilde H_0(t,z)$ invertible on $\Omega$. 
The spectrum of the matrix $\tilde H_0(t,z)$ consists of two distinct eigenvalues of constant multiplicity which do not cross on 
$\Omega$.
\item[e)]
The eigenvalues satisfy a polynomial gap condition at infinity, in the sense that there exist constants $c_0,n_0,r_0>0$ such that
\begin{equation}\label{hyp:gapinfinity}
|h_1(t,z)-h_2(t,z)| \ge c_0 \langle z\rangle^{-n_0}\ \text{for all}\ (t,z)\ \text{with}\ |z|\ge r_0,
\end{equation}
where we denote $\langle z\rangle = (1+|z|^2)^{1/2}$.
\end{enumerate}
\end{assumption}

In the above setting, the trace-free smooth matrix $\tilde H_0(t,z)$ has non-crossing and thus smooth eigenvalues. 
The eigenprojectors of $\tilde H_0(t,z)$ are smooth and are also those of $H(t,z)$. 
Note that one can then modify the function $f$ in $\Omega$ so that the functions 
 \begin{equation}\label{def:f}
 h_j(t,z)= v(t,z) - (-1)^j f(t,z) ,\quad j\in\{1,2\},
 \end{equation}
 are the two smooth eigenvalues of the matrix $H(t,z)$, with smooth associated eigenprojectors 
 $\Pi_1(t,z)$ and $\Pi_2(t,z)$.
 We shall choose $f$ in that manner throughout the paper.

\begin{example}
Take $N=2$, $v,f\in{\mathcal C}^\infty(\R^{2d+1}, \R)$ and  $u\in{\mathcal C}^\infty(\R^{2d+1}, \R^3)$ with  
$|u(t,z)| = 1$ for all $(t,z)$. Consider the Hamiltonian 
$$H(t,z)=v(t,z){\rm Id} + f(t,z) \begin{pmatrix} u_1(t,z) & u_2 (t,z) +iu_3(t,z) \\u_2(t,z) -iu_3(t,z) &-u_1 (t,z) \end{pmatrix}.$$
The smooth eigenvalues of~$H$, 
$h_1 = v+f$ and  $h_2 = v-f$,
cross on the set $\Upsilon = \{f=0\}$, and $H$ satisfies Assumption~\ref{hyp:codim1} as soon as the conditions~\eqref{hyp:transversality}  and \eqref{hyp:gapinfinity} hold. 
\end{example}

Note that the condition~(\ref{hyp:transversality}) implies the transversality of the classical trajectories to the crossing set~$\Upsilon$. The gap condition at infinity~\eqref{hyp:gapinfinity} ensures, that the derivatives 
of the eigenprojectors $\Pi_j(t)$, $j=1,2$, grow at most polynomially, in the sense that for all $\beta\in\N_0^{2d}$ there exists a constant $C_\beta>0$ such that
\begin{equation}\label{bound:projector}
\|\partial_z^{\beta} \Pi_j(t,z)\| \le C_\beta \langle z\rangle^{|\beta|(1+n_0)}\ \text{for all}\ (t,z)\ \text{with}\ |z|\ge r_0,
\end{equation}
see \cite[Lemma~B.2]{CF11} for a proof of this estimate. 

\medskip
We associate with each eigenvalue $h_j$ the classical quantities introduced in section~\ref{sec:defclas}, that we index by~$j$: $\Phi_j^{t,t_0}$, $S_j(t,t_0)$, $F_j(t,t_0)$, etc.
We consider initial data at time~$t=t_0$ as in~\eqref{initialdata}, where the coherent state is associated with the first eigenvalue~$h_1$ and centered in a phase space point $z_0$ such that $(t_0,z_0)\notin \Upsilon$, while $z\mapsto \vec V_0(z)$ is a smooth map with $\|\vec V_0(z)\|=1$ for all $z$.
We assume that the Hamiltonian trajectory  $z_1(t,t_0,z_0)= \Phi^{t,t_0}_1(z_0)$ reaches~$\Upsilon$  at time $t=t^\flat$ and point~$z=z^\flat$ where~\eqref{hyp:transversality} holds. Therefore, $f(t,z)=0$ is a local equation 
of~$\Upsilon$ in a neighborhood $\Omega$ of $(t^\flat,z^\flat)$, and the assumption~\ref{hyp:codim1} implies  
\[\frac{d}{dt}f(t, z_1(t,t_0))\neq 0\]
close to $(t^\flat, z^\flat)$, and guarantees that the trajectory~$z_1(t,t_0,z_0)$ passes through~$\Upsilon$. 
The same holds for trajectories $\Phi^{t,t_0}_1 (z)$ starting from~$z$ close enough to $z_0$.
 
 \medskip 
 
We associate with  $\vec V_0(z)$ the time-dependent eigenvector $(\vec V_1(t,z))_{t\geq t_0}$  constructed as in  Proposition~\ref{prop:eigenvector} for the eigenvalue $h_1(t,z)$ with initial data~$\vec V_0(z)$ at time $t_0$. We also consider the   time-dependent eigenvector  $(\vec V_2(t,z))_{t\geq t^\flat}$ constructed  for $t\ge t^\flat$ as in  Proposition~\ref{prop:eigenvector} for the eigenvalue $h_2(t,z)$ and with initial data at time $t^\flat$ satisfying 
\begin{equation}\label{def:V2}
 \vec V_2(t^\flat, z)= -\gamma (t^\flat,z)^{-1} {\Pi_2(\partial_t \Pi_2 +\{v,\Pi_2\} )\vec V_1}
 (t^\flat ,z)
\end{equation}
\[{\rm with }\;\;
\gamma(t^\flat,z) =\|\left(\partial_t \Pi_2+\{v,\Pi_2\}\right)\vec V_1(t^\flat,z) \| _{\C^N} .
\]
Note that the vector  $\vec V_2(t^\flat,z)$ is in the range of $\Pi_2(t^\flat,z)$.
We next introduce a family of  transformations, which describes the non-adiabatic effects for 
a wave packet that passes the crossing. For $(\mu,\alpha,\beta)\in\R\times\R^{2d}$ and $\varphi\in\Sch(\R^d)$, we set 
  \beq\label{transf2}
{\mathcal T}_{\mu,\alpha, \beta}\varphi(y) =
 \left(\int_{-\infty}^{+\infty}{\rm e}^{i\mu s^2}{\rm e}^{is(\beta\cdot y-\alpha\cdot D_y)}ds\right)\varphi(y).
\eeq
By the Baker-Campbell-Hausdorff formula, we have 
$${\rm e}^{is\beta\cdot y}{\rm e}^{-is\alpha\cdot D_y}= {\rm e}^{is\beta\cdot y-is\alpha\cdot D_y+is^2\alpha\cdot\beta/2},$$ 
and we deduce the equivalent representation
\beq\label{transf1}
 {\mathcal T}_{\mu,\alpha, \beta}\varphi(y) =
  \int_{-\infty}^{+\infty}{\rm e}^{i(\mu-\alpha\cdot\beta/2)s^2} {\rm e}^{is\beta\cdot y}\varphi(y-s\alpha)ds .
 \eeq
 We prove in Proposition~\ref{prop:T} below that this operator maps $\mathcal S(\R^d)$ into itself if and only if $\mu\not=0$. Moreover, for $\mu\not=0$, it is a metaplectic transformation of the Hilbert space
 $L^2(\R^d)$, multiplied by   a complex number. 
In particular,  for any Gaussian function
  $g^\Gamma$, the function   
  $ {\mathcal T}_{\mu,\alpha, \beta}g^\Gamma $ is a Gaussian:
  \[ 
 {\mathcal T}_{\mu,\alpha, \beta}\,g^\Gamma = c_{\mu,\alpha,\beta,\Gamma}\, g^{\Gamma_{\mu, \alpha,\beta,\Gamma}},
 \]
   where $\Gamma_{\mu, \alpha,\beta,\Gamma}\in\mathfrak S^+(d)$ and $c_{\mu,\alpha,\beta,\Gamma}\in \C$ are given in Proposition~\ref{prop:T}.

\medskip 
Combining the parallel transport for the eigenvector and the metaplectic transformation for the non-adiabatic transitions, we obtain the following result. 

\begin{theorem}[Propagation through a codimension one crossing]\label{theo:WPcodim1}
Let Assumption~\ref{hyp:codim1} on the Hamiltonian matrix $H(t)$ hold, and assume that the initial data 
$(\psi^\eps_0)_{\eps>0}$ are wave packets as in~(\ref{initialdata}).
  Let $T>0$ be such that the interval $[t_0,t^\flat] $ is strictly included in the interval $[t_0,t_0+T]$. Then, for all $k\in\N$ there exists a constant $C>0$ such that for all $t\in  [t_0,t^\flat)\cup(t^\flat,t_0+T]$ and for all $\eps\leq [t-t^\flat|^{9/2}$,
$$\left\| \psi^\eps(t) - \widehat {\vec V}_1(t) v^\eps_1(t) -\sqrt\eps {\bf 1}_{t>t^\flat} \widehat {\vec V}_2(t) v^\eps_2(t) \right\|_{\Sigma^k_\eps}\leq C \,\eps^{m},$$
with an exponent $m\ge 5/9$. The components of the approximate solution are 
\[
v^\eps_1(t) = \U_{h_1}^\eps(t,t_0) v^\eps_0\quad \text{and}\quad v^\eps_2(t) =\U_{h_2}^\eps(t,t^\flat)v^\eps_2(t^\flat)
\]
with 
\begin{equation}\label{eq:v2}
v^\eps_2(t^\flat)= \gamma^\flat {\rm e}^{iS^\flat/\eps}\wp^\eps_{z^\flat}{\mathcal T}^\flat
\varphi_1(t^\flat),
\end{equation}
where $\varphi_1(t) = {\mathcal M}[F_{h_1}(t,t_0,z_0)]\varphi_0$
is the leading order profile of the coherent state $v^\eps_1(t)$ given by Proposition~\ref{wpevol1}, and 
\begin{equation} \label{def:alpha} 
\gamma^\flat = \gamma(t^\flat,z^\flat) = \|\left(\{v,\Pi_2\}+\partial_t\Pi_2\right)\vec V_1(t^\flat,z^\flat) \| _{\C^N} .
\end{equation}
The transition operator
\beq\label{def:T_sharp}
\mathcal T^\flat= \mathcal T_{ \mu^\flat, \alpha^\flat,\beta^\flat}
\eeq
is defined by the parameters
\begin{equation}\label{def:lambda}
\mu^\flat = \tfrac{1}{2}\left( 
\partial_t f+\{v,f\}\right)(t^\flat, z^\flat)\;\; \text{and}\;\;
(\alpha^\flat,\beta^\flat) = Jd_zf(t^\flat, z^\flat). 
\end{equation}
The constant $C = C(T,k,z_0,\varphi_0)>0$ is $\eps$-independent but depends on the 
Hamiltonian $H(t,z)$, the final time $T$, and on the initial wave packet's center~$z_0$ and profile~$\varphi_0$.
\end{theorem}

Note that by Assumption~\ref{hyp:codim1}, $\mu^\flat\not=0$,  
which guarantees that ${\mathcal T}^\flat
\varphi_1(t^\flat)$ is Schwartz class. Besides, if the Hamiltonian is time-independent, then Assumption~\ref{hyp:codim1} also implies that $(\alpha^\flat,\beta^\flat)\not=(0,0)$.
The coefficient $\gamma^\flat$ quantitatively describes the distortion of the projector~$\Pi_1$ during its evolution along the flow generated by $h_1(t)$.
In particular, we have 
$$\gamma^\flat =\|\left(\{v,\Pi_2\}+\partial_t\Pi_2\right)\vec V_1(t^\flat,z^\flat) \| _{\C^N} =\|\left(\{v,\Pi_1\}+\partial_t\Pi_1\right)\vec V_1(t^\flat,z^\flat)\|_{\C^N}.$$
Moreover, if the matrix $H$ is diagonal (or diagonalizes in a fixed orthonormal basis that is $(t,z)$-independent), then  $\gamma^\flat=0$: the equations are decoupled (or can be decoupled), and one can then apply the result for a system of two independent equations with a scalar Hamiltonian and, of course, there is no interaction between the modes.

\medskip 

The proof uses two types of arguments, one of the them applying away from the crossing set $\Upsilon$, and the other one in a boundary layer of~$\Upsilon$. The boundary layer is taken of size $\delta>0$, and we have to balance the two estimates: an error of order $\eps \delta^{-2}$ which comes from the adiabatic propagation of wave packets outside the boundary layer, and an additional error of order $\delta\eps^{1/3}$ generated by the passage through the boundary. The choice of $\delta=\eps^{2/9}$ optimizes the combined estimate and 
yields convergence of order $\eps^m$ with $m\ge 5/9$. We also want to emphasize that the method of proof we propose here allows to systematically avoid the  impressive computations, which appear in~\cite{Hag94} pages~65 to~72, and are also present in~\cite{WW} via the reference [46] to which the authors refer therein. 
  
\medskip 
The wave packet that makes the transition to the other eigenspace can be described even more explicitly for the special case that the initial wave packet is a Gaussian state. The following corollary is proved in Proposition~\ref{prop:T}.

\begin{corollary}[Transitions for Gaussian wave packets] \label{cor:gaussian} We consider the situation of Theorem~\ref{theo:WPcodim1} and in particular the transition operator $\mathcal T^\flat$ defined by the parameters $\mu^\flat\neq 0$ and $(\alpha^\flat,\beta^\flat)\in\R^{2d}$. 
\begin{enumerate}\setlength{\itemsep}{0.5ex}
\item
If $v_0^\eps=\mathcal{WP}^\eps_{z_0} (g^{\Gamma_0})$ is a Gaussian state with width matrix 
$\Gamma_0\in{\mathfrak S}^+(d)$, then 
$$
v^\eps_2(t^\flat)= \gamma^\flat \sqrt{2\pi\over i\mu^\flat}\, {\rm e}^{iS^\flat/\eps}\,
\wp^\eps_{z^\flat}( g^{\Gamma^\flat})\ech
$$
$$ \mbox{with} \;\;
\Gamma^\flat = \Gamma_1(t^\flat,t_0,z_0)  -\frac{(\beta^\flat-\Gamma_1(t^\flat,t_0,z_0)\alpha^\flat)\otimes (\beta^\flat-\Gamma_1(t^\flat,t_0,z_0)\alpha^\flat) }{2\mu^\flat-\alpha^\flat\cdot\beta^\flat+\alpha^\flat\cdot\Gamma_1(t^\flat,t_0,z_0)\alpha^\flat}$$
and $\Gamma_1(t^\flat,t_0,z_0)$ is the image of $\Gamma_0$ by the flow map associated with $h_1(t,z)$ by~\eqref{def:Gamma}.
\item
If $A\in \mathcal C^\infty (\R^{2d})$ is a polynomial function and 
$v_0^\eps=\mathcal{WP}^\eps_{z_0} ({\rm op}^w_1(A)g^{\Gamma_0})$, then 
$$
v^\eps_2(t^\flat)= \gamma^\flat \sqrt{2\pi\over i\mu^\flat}\, {\rm e}^{iS^\flat/\eps}\,
\wp^\eps_{z^\flat}\!\left({\rm op}^w_1(A^\flat ) g^{\Gamma^\flat}\right)
$$
with 
$A^\flat= A\circ \Phi_{\alpha^\flat,\beta^\flat}(-(4\mu^\flat)^{-1})$ where $\Phi_{\alpha^\flat,\beta^\flat}(t)$ is the symplectic $2d\times 2d$ matrix given by
\begin{align} \label{def:Phi}
 &\qquad \qquad \Phi_{\alpha^\flat,\beta^\flat}(t) =\begin{pmatrix}  \1 -2t\beta^\flat\otimes\alpha^\flat&2t\alpha^\flat\otimes\alpha^\flat, \\-2t\beta^\flat\otimes\beta^\flat & \1 +2t\alpha^\flat\otimes\beta^\flat\end{pmatrix}.
\end{align}
\end{enumerate}
\end{corollary}

As a concluding remark of this section, we want to emphasize that our results indeed generalize those of~\cite{Hag94,WW}. 
\begin{enumerate}\setlength{\itemsep}{0.5ex}
\item In the Schr\"odinger example~\eqref{ex:hag}, denoting by  $E_A(x)$ and $E_B(x)$  the  two eigenvalues of the potential matrix $V(x)$\ech as in~\cite{Hag94}, one has 
\[\alpha^\flat_S=0 ,\;\;\beta^\flat_S= \nabla (E_A-E_B)(q^\flat),\;\;
\mu^\flat_S= p^\flat\cdot \nabla (E_A-E_B)(q^\flat).
\]
These coefficients appear in equation~(5.3) of~\cite{Hag94}. There, the initial states 
are Gaussian wave packets that are multiplied 
with a polynomial function. Thus, the second part of Corollary~\ref{cor:gaussian} reproduces 
these results. 
\item 
For the Bloch example~\eqref{ex:WW}, we obtain 
\[\alpha^\flat_A=\nabla (E_+-E_-)(p^\flat) ,\;\;\beta^\flat_A= 0,\;\;\mu^\flat_A = -\frac 1 2 \nabla W(q^\flat) \cdot \nabla (E_+-E_-)(p^\flat), \]
 where $E_\pm(\xi)$ are the eigenvalues of $A(\xi)$ as in equation~(3.41) of~\cite{WW}. 
 The result of \cite[Theorem~3.20 (via Definition~3.18)]{WW} is therefore a special case of ours. 
  \end{enumerate}

We notice, that for these special examples either one of the coefficients $\alpha^\flat$ or $\beta^\flat$ is~$0$.
This need not be the case for more general Hamiltonians that have position and momentum variables 
mixed in the matrix part of the Hamiltonian. Actually, for Dirac Hamiltonians with electromagnetic potential~$(V,A)$, the function $\xi-A(t,x)$ appears in the coefficients of the matrix. Also for the propagation of acoustical waves in elastic media the Hamiltonian is of the form $\rho(x)\1_{\C^N} -\Gamma(x,\xi)$,where $\rho(x)>0$ is the density 
and $\Gamma(x,\xi)$ the elastic tensor.

\subsection{Organization of the paper} 
 The proof of Theorem~\ref{theo:WPcodim1} is decomposed into two steps: an analysis outside the crossing region in Section~\ref{sec:outside} and an analysis in the crossing region in Section~\ref{subsec:inside}, that allows to conclude the proof in Section~\ref{subsec:proof}, together with the one of  Corollary~\ref{cor:gaussian}. 
Finally, we gather in four Appendices various results about wave packets, algebraic properties of the projectors and parallel transport, analysis of the transfer operators $\mathcal T_{\mu,\alpha,\beta}$, and technical computations.




\section{Adiabatic decoupling outside the crossing region}\label{sec:outside}

In this section,  we consider a family of solutions to equation~(\ref{system}) in the case where the Hamiltonian $H(t,z)$ satisfies Assumption~\ref{hyp:codim1}  and  with an initial datum which is a coherent state  as in~\eqref{initialdata}.
We focus here on regions where the classical trajectories associated with the coherent state do not touch the crossing set $\Upsilon$ but are close enough. We prove the next adiabatic result. 

\begin{proposition}\label{prop:propagationt**}
Let $k\in\N$, $\delta = \delta(\eps) $ be such that $ \sqrt\eps\ll \delta\leq 1$. 
Let $f(t,z)=0$ be an equation of~$\Upsilon$ in an open set $\Omega\subset \R\times\R^{2d}$. 
 Assume that  for $j\in\{1,2\}$, 
 $$u^\eps_j=\mathcal{WP}^\eps_{\widetilde z_j} (\widetilde \varphi_j),$$
where $\widetilde \varphi_1,\, \widetilde \varphi_2\in\mathcal S(\R^d)$, $\widetilde z_1,\,\widetilde z_2\in \R^d$  are such that
 there exist $s_1,s_2\in\R $, $c,C>0$  such that for all $j\in\{1,2\}$ and 
$ t\in [s_1,s_2]$, $z_j(t):=\Phi_{j}^{t,s_1}(\widetilde z_j)\in\Omega$ with 
$|f(z_j(t))|>c\delta$
 and 
$$\left\|\psi^\eps(s_1)-\widehat {\vec V_1(s_1)} u^\eps_1 - \widehat {\vec V_2(s_1)} u^\eps_2\right\|_{\Sigma^k_\eps}\leq C\eps.$$
 Then, there exists $C_k>0$ such that for all $j\in\{1,2\}$, 
$$\sup_{t\in [s_1,s_2]} \left\| \widehat \Pi_j \psi^\eps(t) - \widehat {\vec V_j(t)} {\mathcal U}^\eps_{h_j}(t,s_1)
u^\eps_j \right\| _{\Sigma^k_\eps} \leq C_k\,\eps\, \delta^{-2}.$$
The constant $C_k$ does not depend on $\delta$ and $\eps$.
\end{proposition}

For fixed $\delta$, that is independent of $\eps$, this Proposition implies Theorem~\ref{theo:WPcodim1} for $t\in[0,t^\flat[$. We shall choose later $\delta=\eps^{1/3}$  for obtaining a global a priori estimate in  Section~\ref{sec:apriori} below. Finally with $\delta=\eps^{2/9}$, we will prove Theorem~\ref{theo:WPcodim1}  in Section~\ref{subsec:proof} by using the Proposition~\ref{prop:propagationt**}  for propagation times $t \in [t_0, t^\flat-\delta]$ and $t\in [t^\flat+\delta, t^\flat +T]$ with initial data at times $t=t_0$ and $t=t^\flat +\delta$ respectively.  

\begin{remark}\label{rem:generalisation2}
Pushing the construction of superadiabatic projectors of Appendix~\ref{App:projectors}, we would obtain that $\psi^\eps(t)$ can be approximated by an asymptotic sum of wave packets up to order $\eps^N \delta^{p(N)}$ for some $p(N)\leq N$ to be computed precisely. 
\end{remark}

\subsection{Proof of the adiabatic decoupling}

We prove here Proposition~\ref{prop:propagationt**}.
\begin{proof}
Because of the linearity of the equation, it is enough to assume that the contribution of $\psi^\eps(s_1)$ on one of the modes is negligible at the initial time~$s_1$. The roles of the two modes being symmetric, we can choose equivalently one or the other one. Therefore, without loss of generality, we assume
$\psi^\eps(s_1) = \widehat{\vec V_1(s_1)} u^\eps_1$,
and we focus on  
$$\psi^\eps_{1,{\rm app}}(t) := \widehat{\vec V_1(t)} \mathcal U^\eps_{h_1} (t,s_1) u^\eps_1$$
Then, using the parallel transport equation~\eqref{eq:eigenvector} associated with the eigenvalue $h_1$, 
\begin{align}\nonumber
i\eps \partial_t \psi^\eps_{1,{\rm app}}(t) &= \widehat h_1 \psi^\eps_{1,{\rm app}}(t) + 
 \left( \left[ \widehat {\vec V_1(t)} , \widehat h_1 \right] +i\eps \partial_t \widehat{ \vec V_1(t)}\right)\mathcal U^\eps_{h_1} (t,s_1) u^\eps_1 \\
 \label{eq:psi1app}
& =(\widehat h_1{\rm Id} + \eps \widehat  \Theta_1) \psi^\eps_{1,{\rm app}}(t) +
\eps^2\widehat{r(t)}\,\mathcal U^\eps_{h_1} (t,s_1) u^\eps_1,
\end{align}
where the remainder $r(t)$ depends on second order derivatives of $h_1$ and $\vec V_1$. Since 
$u^\eps_1$ is a wave packet with a Schwartz function amplitude, we obtain
\begin{equation}\label{eq:approxsol}
i\eps \partial_t \psi^\eps_{1,{\rm app}}(t) = (\widehat h_1{\rm Id} + \eps \widehat  \Theta_1) \psi^\eps_{1,{\rm app}}(t) + O(\eps^2)
\end{equation}
in $\Sigma^k_\eps$ for all $k\in \N$. 
\medskip

We now use the superadiabatic correctors of $\Pi_1$ and $\Pi_2$ defined in Definition~\ref{def:superadiabatic} that we denote  by $\mathbb P_1$ and $\mathbb P_2$, respectively, and the associated correctors $\Theta_1$ and $\Theta_2$ of the Hamiltonian $H$. Since $\mathbb P_1$ and $\mathbb P_2$ are singular on~$\Upsilon$, we use cut-off functions that follow the flows arriving at time $s_2$ in $\Phi^{s_2,s_1} _{h_1}(\tilde z_1)$. We introduce two sets of cut-off functions, one for each mode. Let $I$ an interval containing $[s_1,s_2]$ and for $j\in\{1,2\}$ let the cut-off functions
 $\chi^\delta_j ,\tilde \chi^\delta_j\in\mathcal C(I,{\mathcal C}_0^\infty(\R^{2d}))$ satisfy as in Lemma~\ref{lem:diagonalisation}:
 \begin{enumerate}
 \item For any $t\in I$ and any $z$ in the support of $\chi^\delta_j(t)$ and $\tilde \chi^\delta_j(t)$ we have
$|f(t,z)|>\delta$.
\item The functions $\chi^\delta_j$ and $\tilde\chi^\delta_j$ are identically equal to $1$ close to a trajectory $\Phi^{t,s_1}_{j}(\tilde z_1)$ for all $t\in I$ and they satisfy 
$$\partial_t \chi^\delta_j + \left\{h_j, \chi^\delta \right\}=0,\;\;\partial_t \tilde\chi^\delta_j + \left\{h_j, \tilde\chi^\delta_j \right\}=0.$$
\item The functions $\tilde\chi^\delta_j$ are supported in $\{\chi^\delta_j=1\}$. 
\item Finally, we require 
$\chi_1^\delta(s_2)=\chi_2^\delta(s_2)$ and $\tilde \chi_1^\delta(s_2)=\tilde \chi_2^\delta(s_2).$
\end{enumerate}
We set for $t\in [s_1,s_2]$ 
$$w^\eps_1 (t)= \widehat {\tilde \chi^\delta_1}(\widehat {\chi^\delta_1 \Pi^\eps_1}\psi^\eps(t)- 
\psi^\eps_{1,{\rm app}}(t)) ,\;\;
w^\eps_2(t)=  \widehat {\tilde \chi^\delta_2}\widehat {\chi^\delta_2 \Pi^\eps_{2}}\psi^\eps(t),$$
$$\mbox{where } \;\;\Pi^\eps_j(t,z)=\Pi_j(t,z)+\eps \mathbb P_j(t,z),\;\;\forall z\in \R^{2d}\setminus \Upsilon,\;\; t\in I, \;\;j\in\{1,2\}.$$
Then, as a consequence of~\eqref{eq:approxsol} and of Lemma~\ref{lem:diagonalisation}, we have for $j\in\{1,2\}$ and in $\Sigma^k_\eps$, 
$$i\eps \partial_t w^\eps_j(t)= (\hat h_j + \eps \widehat \Theta_j) w^\eps_j(t) +O(\eps^2\delta^{-2}).$$
For the initial data at time $t=s_1$, we have in $\Sigma^k_\eps$,
\[
w^\eps_1(s_1) =  \widehat {\tilde \chi^\delta_1}
\big(\widehat {\chi^\delta_1 \Pi^\eps_1}\widehat{\vec V_1}-\widehat{\vec V_1}\big)u^\eps_1 = O(\eps\delta^{-1}),
\quad
w^\eps_2(s_1) =  \widehat {\tilde \chi^\delta_2}\,
\widehat {\chi^\delta_2 \Pi^\eps_2}\,\widehat{\vec V_1}u^\eps_1 = O(\eps\delta^{-1}).
\]
We deduce that for any $k\in\N$, $j\in\{1,2\}$ and  $t\in [s_1,s_2]$, we have in $\Sigma^k_\eps$,
$w^\eps_j(t)=O( \eps \delta^{-2}).$
When $t=s_2$, we have 
\begin{align*}
w_1^\eps(s_2) + w_2^\eps(s_2) 
&= \widehat {\tilde \chi^\delta_1}
\left(\op_\eps(\chi^\delta_1 (\Pi^\eps_1+\Pi^\eps_2))\psi^\eps(s_2) - \psi^\eps_{1,\rm app}(s_2)\right) \\
&= 
 \widehat {\tilde \chi^\delta_1}(\psi^\eps(s_2)- \psi^\eps_{1,\rm app}(s_2)) + O(\eps\delta^{-1})
\end{align*}
and thus
$\displaystyle{
\widehat{\tilde \chi^\delta_1}(s_2)\psi^\eps(s_2)= \widehat{\tilde \chi^\delta_1}(s_2)\psi^\eps_{1,\rm app}(s_2) +O(\eps\delta^{-2}).
}$
Because of the localisation of the wave packet $\psi^\eps_{1,\rm app} (s_2)$, as stated in Remark~\ref{rem:apendix},  we have in $\Sigma^k_\eps$ for any $N\in\N$,
\[
\widehat{\tilde \chi^\delta_1}(s_2)\psi^\eps_{1,\rm app}(s_2) = \psi^\eps_{1,\rm app}(s_2) + 
 O(\eps^{N/2} \delta^{-N}).
\] 
Hence, choosing $N=2$, we obtain
 $$\widehat{\tilde \chi^\delta_1}(s_2)\psi^\eps(s_2)= \psi^\eps_{1,\rm app}(s_2) +O(\eps\delta^{-2}),$$
 and it only remains to study $(1-\widehat{\tilde \chi^\delta_1}(s_2))\psi^\eps(s_2)$. Before that, some remarks are in order. 
 Note that the arguments developed above do not depend on the choice of $s_2$ and could have been developed for any $s\in [s_1,s_2]$. They are also independent of the choices of the functions $\chi^\delta_j$ and 
 $\widetilde \chi^\delta_j$ as long as they satisfy the properties stated above. Therefore, we have actually obtained a more general result, namely that for any function $\theta$ supported in $\{|f|>\delta\}$ and equal to~$1$ close to $\Phi_{1}^{t,s_1}(\widetilde z_1)$, we have for $t\in[s_1,s_2]$,
 \begin{equation}\label{apriori34}
 \widehat{\theta}\psi^\eps(t)= \widehat \theta \psi^\eps_{1,\rm app}(t) +O(\eps\delta^{-2}).
 \end{equation}
 We can now study $(1-\widehat{\tilde \chi^\delta_1}(s_2))\psi^\eps(s_2)$. We set for $s\in[s_1,s_2]$,
 $w^\eps(s)= (1-\widehat{\tilde \chi^\delta_1}(s))\psi^\eps(s).$
 We have 
 \begin{align*}
 i\eps\partial_s w^\eps(s) &= \widehat H(s) w^\eps(s) -\left[ \widehat{\tilde \chi^\delta_1}(s), \widehat H(s)\right] \psi^\eps(s) -i\eps\widehat{\partial_s\tilde\chi^\delta_1(s)}\psi^\eps(s)\\
 &=  \widehat H(s) w^\eps(s)  -\eps \widehat{r^\eps_\delta(s)} \psi^\eps(s) +O(\eps^2\delta^{-2})
 \end{align*}
 where $r^\eps_\delta(s)$ depends linearly on $d\tilde \chi^\delta_1(s)$, and thus is compactly supported close to the trajectory $\Phi_{1}^{t,s_1}(\widetilde z_1)$ and equal to~$0$ very close to it. Therefore, by~\eqref{apriori34}  and  Remark~\ref{rem:apendix},
 $$\eps\,\widehat{r^\eps_\delta(s)} \psi^\eps(s)=\eps\,\widehat{r^\eps_\delta(s)} \psi^\eps_{1,\rm app}(s)=  O(\eps^{N/2+1} \delta^{-N-1})$$
 for any $N\in\N$. 
 Choosing $N=1$, we deduce $w^\eps(s_2)=O(\eps\delta^{-2})$. 
\end{proof}


\subsection{A global a priori estimate} \label{sec:apriori}

In this section, we prove the following a priori estimate. 

\begin{lemma}\label{lem:apriori}
Let $k\in\N$ and $T>0$ such that $[t_0,t^\flat]$ is strictly included in $[t_0,t_0+T]$.  Then 
there exists a constant $C_k>0$ such that 
\begin{equation}\label{apriori}
\sup_{t\in[t_0,t_0+T]} \| \psi^\eps(t)- \widehat { \vec V_1}(t) v^\eps_1(t)\|_{\Sigma^k_\eps} \leq C_k \,\eps^{1/3},
\end{equation}
where $v^\eps_1(t) = {\mathcal U}_{h_1}^\eps(t,t_0)v^\eps_0$ for all $t\in[t_0,t_0+T]$.
\end{lemma}

In the next section, we shall  improve this estimate to go beyond this approximation and exhibits elements of order $\sqrt\eps$. However, we shall use this a priori estimate, together with elements developed in this section.

\begin{proof}
Of course, in view of the results of the preceding section, we choose $\delta>0$ and we focus on the time interval $[t^\flat-\delta,t^\flat+\delta]$, taking into account that for times $t\in [t_0, t^\flat-\delta]$, we have 
$$ \| \psi^\eps(t)- \widehat { \vec V_1}(t) v^\eps_1(t)\|_{\Sigma^k_\eps} \leq C_k\,  \eps\delta^{-2}$$
for some constant $C_k>0$, and that for $t\in[t^\flat +\delta, t_0+T]$ we can use the same kind of transport estimate since the trajectory does not meet again the crossing set. It is thus enough to pass from $t^\flat -\delta$ to $t^\flat +\delta$ and  analyze $\psi^\eps(t^\flat+\delta)$.
Between the times $t^\flat-\delta$ and $t^\flat +\delta$, 
 we cannot use the super-adiabatic corrections  to the projectors $\Pi_1$ and $\Pi_2$, because they become 
 singular when the eigenvalue gap closes. We thus simply work with the projectors $\Pi_1$ and $\Pi_2$. 
We define the families $w^\eps(t)=(w^\eps_1(t),w^\eps_2(t))$ by 
\begin{equation}\label{def:tildeweps} 
w_1^\eps=
\widehat \Pi_1 \psi^\eps -  \widehat {\vec V_1} v^\eps_1,\;\;
w^\eps_2 = \widehat \Pi_2 \psi^\eps.
\end{equation}
Since $\psi^\eps(t)$ and $\widehat {\vec V_1} v^\eps_1(t)$ are in all spaces $\Sigma^{\ell}_\eps(\R^d)$ for $\ell\in \N$ and $t\in [t_0,t_0+T]$, the same is true for $w_1^\eps(t)$ and $w_2^\eps(t)$. We now use our former observations, that is, the evolution equation~\eqref{eq:psi1app} for the approximate wave packet and the relation~\eqref{preliminary_calcul} of Appendix~\ref{App:projectors}, which gives that $w^\eps(t)$ satisfies the following system:
$$\left\{\begin{array}{rl}
i\eps \partial_t w^\eps_1 &= \widehat h_1w^\eps _1 + i\eps f^\eps_1,\\*[1ex] 
i\eps \partial_t  w^\eps_2 &= \widehat h_2 w^\eps _2 + 
\tfrac{i\eps}{2} \widehat{B_2\Pi_1} \widehat{\vec V_1} v_1 + i\eps f^\eps_2 
\end{array}\right.$$
with
\begin{equation}\label{def:feps}
f^\eps_1 = -i\widehat\Theta_1 w^\eps_1 + \tfrac12 \widehat{B_1\Pi_2} w^\eps_2 + \eps r^\eps_1\quad \text{and}\quad
f^\eps_2 = -i\widehat\Theta_2 w^\eps_2 + \tfrac12 \widehat{B_2\Pi_1} w^\eps_1 + \eps r^\eps_2.
\end{equation}
The matrices $B_1$ and $B_2$ are defined according to
\[
B_j = -2\partial_t\Pi_j  - \{h_j,\Pi_j\} + \{\Pi_j,H\},\qquad j=1,2,
\]
and the sequences $(r^\eps_1(t))_{\eps>0}$ and $(r^\eps_2(t))_{\eps>0}$ are uniformly bounded in $\Sigma^k_\eps(\R^d)$ due to the polynomial growth estimate \eqref{bound:projector} for the eigenprojectors. We immediately deduce that for all $t\in[t_0,t_0+T]$, 
\begin{align}\nonumber
w^\eps_1(t) &=\  {\mathcal U}_{h_1}^\eps(t,t^\flat-\delta) w^\eps_1(t^\flat-\delta) + \int_{t^\flat-\delta}^t {\mathcal U}_{h_1}^\eps(t,\sigma) f^\eps_1(\sigma) d\sigma  ,\\*[-2ex]
\label{eq:tildeweps}\\*[-2ex]
\nonumber
w^\eps_2(t)  &=\  {\mathcal U}_{h_2}^\eps(t,t^\flat-\delta) w^\eps_2(t^\flat-\delta) + \int_{t^\flat-\delta}^t  {\mathcal U}_{h_2}^\eps(t,\sigma) f^\eps_2(\sigma) d\sigma \\
\nonumber
& \qquad +{1\over 2}  \int_{t^\flat-\delta}^t {\mathcal U}_{h_2}^\eps(t,\sigma)
\widehat {B_2 \Pi_1} \widehat {\vec V_1}(\sigma) v_1^\eps(\sigma)d\sigma.
\end{align}
Therefore, in $\Sigma^k_\eps(\R^d)$, for all times $t\in [t^\flat-\delta,t^\flat+\delta]$ and $j\in\{1,2\}$,
$ w^\eps_j(t)= O ( \eps\delta^{-2}) +O(\delta) .$
Choosing $\delta=\eps^{1/3}$, we obtain that $ w^\eps_j(t)= O (\eps^{1/3})$. 
\end{proof}

\section{Analysis in the crossing region}\label{subsec:inside}

We now want to pass through the crossing and derive a more precise estimate on the function~$\psi^\eps(t^\flat +\delta)$. We prove the following result. 

\begin{proposition}\label{prop:throughcrossing}
Assume $\sqrt\eps\ll\delta\ll \eps^{1/3}$. Then, for all $k\in\N$, there exists a constant $C_k>0$ such that 
$$
\left\|\psi^\eps(t^\flat +\delta) 
- \widehat{\vec V_1}(t^\flat +\delta)  v^\eps_1(t^\flat +\delta)
-  \sqrt\eps \widehat{\vec V_2}(t^\flat +\delta)  v^\eps_2(t^\flat +\delta ) \right\|_{\Sigma^k_\eps} \leq 
C_k(\eps \delta^{-2} + \eps^{1/3}\delta), 
$$
where $v^\eps_1(t) = \U^\eps_{h_1}(t,t_0)v^\eps_0$ and  $v^\eps_2(t)=\U^\eps_{h_2}(t,t^\flat)v^\eps_2(t^\flat)$ are as in Theorem~\ref{theo:WPcodim1}.
\end{proposition}

\begin{proof}
We split the proof in several steps. 
In Lemma~\ref{lem:step1} we use the a priori estimate of Lemma~\ref{lem:apriori} to simplify the approximation of $\psi^\eps(t^\flat + \delta)$ and exhibit the contribution of order $\sqrt\eps$ according to 
$$
\psi^\eps(t^\flat+\delta)= \widehat{\vec V_1}(t^\flat+\delta)v^\eps_1(t^\flat+\delta) + 
{\rm e}^{iS^\flat /\eps}\, \U_{h_2}^\eps(t^\flat+\delta,t^\flat)\,  A_\eps + O(\eps\delta^{-2})  +O(\eps^{1/3} \delta).
$$
Then, we carefully analyze the contribution $A_\eps$ and construct a preliminary transfer operator $\mathcal T^\eps$
satisfying
\[
A_\eps= \wp^\eps_{z^\flat} {\mathcal T}^\eps \varphi_1(t^\flat) + O(\sqrt\eps \delta),
\]
see Lemma~\ref{lem:trans}. As the third step, Lemma~\ref{lem:transap} establishes the relation to the transfer operator $\mathcal T^\flat$ according to 
\[ 
{\mathcal T}^\eps = \sqrt\eps \,\mathcal{Q}^\eps(0) \mathcal T^\flat + O(\sqrt \eps \delta) + O(\eps\delta^{-1})
\]
with ${\mathcal Q}^\eps(0) = \op_1^w((\gamma\vec V_2)(t^\flat,z^\flat + \sqrt\eps\bullet))$. 
  The wave packet relation~\eqref{w:pseudo} in combination with symbolic calculus 
  implies for all $\varphi\in\mathcal S(\R^d)$ that
  \begin{align*}
  \mathcal{WP} ^\eps_{z^\flat} {\mathcal Q}^\eps(0)\varphi &= \widehat {\gamma\vec V_2}(t^\flat) 
  \mathcal{WP} ^\eps_{z^\flat}\varphi \\
  &= \widehat{\vec V_2}(t^\flat) \widehat{\gamma}(t^\flat)\mathcal{WP} ^\eps_{z^\flat}\varphi + O(\eps) 
  =  \widehat{\vec V_2}(t^\flat) \gamma^\flat\,\mathcal{WP} ^\eps_{z^\flat}\varphi + O(\sqrt\eps).
  \end{align*}
  Hence, we have proven that
 \begin{align*}
 &{\rm e}^{iS^\flat /\eps}\, \U_{h_2}^\eps(t^\flat+\delta,t^\flat)\,  A_\eps \\
 &= 
 \sqrt\eps\, \gamma^\flat\,{\rm e}^{iS^\flat /\eps}\, \U_{h_2}^\eps(t^\flat+\delta,t^\flat) \widehat {\vec V_2}(t^\flat) 
 \wp^\eps_{z^\flat} {\mathcal T}^\flat \varphi_1(t^\flat) + O(\eps\delta^{-2})  +O(\eps^{1/3} \delta)\\
 &=\sqrt\eps\, \U_{h_2}^\eps(t^\flat+\delta,t^\flat)\, \widehat {\vec V_2}(t^\flat) v^\eps_2(t^\flat) + 
 O(\eps\delta^{-2})  +O(\eps^{1/3} \delta).
 \end{align*}
It remains to analyze the function $\omega(t) = \vec V_2(t)\U^\eps_{h_2}(t,t^\flat)-\U^\eps_{h_2}(t,t^\flat)\vec V_2(t^\flat)$. An analogous calculation to the one at the beginning of the proof of Proposition~\ref{prop:propagationt**} yields that
\[
i\eps\partial_t\omega = \widehat h_2\omega + O(\eps).
\]
Since $\omega(t^\flat) = 0$, the Duhamel principle implies that $\omega(t^\flat+\delta) = O(\delta)$ and 
\[
{\rm e}^{iS^\flat /\eps}\, \U_{h_2}^\eps(t^\flat+\delta,t^\flat)\,  A_\eps = 
\sqrt\eps\,  \widehat {\vec V_2}(t^\flat+\delta) \U_{h_2}^\eps(t^\flat+\delta,t^\flat)\, v^\eps_2(t^\flat) + 
 O(\eps\delta^{-2})  +O(\eps^{1/3} \delta).
\]
\end{proof}

\subsection{Using the a priori estimate}
We start describing the part of the wave packet that has been transferred at the crossing and 
identify its main contribution.

\begin{lemma} \label{lem:step1}
Let $k\in\N$. With the assumptions of Proposition~\ref{prop:throughcrossing}, we have in $\Sigma^k_\eps(\R^d)$,
$$\psi^\eps(t^\flat+\delta)= \widehat{\vec V_1}(t^\flat+\delta)v^\eps_1(t^\flat+\delta) + 
{\rm e}^{iS^\flat /\eps}\, \U_{h_2}^\eps(t^\flat+\delta,t^\flat)\,  A_\eps + O(\eps\delta^{-2}) +O(\eps^{1/3} \delta)$$ 
\begin{equation}\label{def:Aeps}
\mbox{with}\;\;A_\eps = \int_{t^\flat-\delta}^{t^\flat +\delta}  \U_{h_2}^\eps(t^\flat,\sigma)
\widehat{\gamma\vec V_2}(\sigma)\, \U_{h_1}^\eps(\sigma, t^\flat) 
{\mathcal{WP}}_{z^\flat} \varphi_1(t^\flat) d\sigma,
\end{equation}
where the eigenvector $\vec V_2$ is defined in \eqref{def:V2} and the Schwartz function $\varphi_1(t^\flat)$ is associated with the profile $\varphi_0$ of the initial wave packet according to Proposition~\ref{wpevol1}.
\end{lemma}

\begin{proof}
We again analyse the functions $w^\eps_1(t)$ and $w^\eps_2(t)$ introduced in~\eqref{def:tildeweps}, that are of order $\eps\delta^{-2}$ at time $t=t^\flat-\delta$.  By the a priori estimate of Lemma~\ref{lem:apriori}, the remainder terms $f^\eps_1(t)$ and $f^\eps_2(t)$, which appear in~\eqref{def:feps}, are  of order $\eps^{1/3}$. Therefore,  the relation~\eqref{eq:tildeweps} gives  for all times $t\in[t^\flat-\delta,t^\flat+\delta]$ and in $\Sigma^k_\eps(\R^d)$,
\begin{align*} 
w^\eps_1(t) & = O(\eps \delta^{-2})  + O(\delta \eps^{1/3}) ,\\
w^\eps_2(t)  &=  O(\eps \delta^{-2})  + O(\delta\eps^{1/3})  + 
{1\over 2}  \int_{t^\flat-\delta}^t {\mathcal U}_{h_2}^\eps(t,\sigma)
\widehat {B_2 \Pi_1} \widehat {\vec V_1}(\sigma) v_1^\eps(\sigma)d\sigma.
\end{align*}
At this stage of the proof, we write $B_2\Pi_1 = \Pi_1B_2\Pi_1 + \Pi_2B_2\Pi_1$ and take advantage of $\displaystyle{\Pi_1B_2\Pi_1=(h_2 -h_1)\Pi_1 \{\Pi_1,\Pi_1\} \Pi_1}$ (see Lemma~\ref{lem:B})
to write  
\begin{align*}
&\int_{t^\flat-\delta}^t  {\mathcal U}_{h_2}^\eps(t,\sigma)
\widehat {\Pi_1 B_2 \Pi_1} \widehat {\vec V_1}(\sigma)\,  
{\mathcal U}^\eps_{h_1}(\sigma,t^\flat-\delta) v^\eps_1(t^\flat-\delta) d\sigma \\
&= 
i\eps \int_{t^\flat-\delta}^t  {d\over d\sigma} \left(
 {\mathcal U}_{h_2}^\eps(t,\sigma)\, {\rm op}_\eps^w\!\left(
\Pi_1 \{\Pi_1,\Pi_1\} \Pi_1\vec V_1(\sigma)\right)  
{\mathcal U}_{h_1}^\eps(\sigma,t^\flat-\delta)\right) v_1^\eps(t^\flat-\delta)\,d\sigma 
+ \eps\rho^\eps(t) \\*[1ex]
&=  \eps \tilde\rho^\eps(t), 
\end{align*}
where both families $(\rho^\eps(t))_{\eps>0}$ and $(\tilde\rho^\eps(t))_{\eps>0}$ are uniformly bounded in $\Sigma^k_\eps(\R^d)$. 
Therefore, 
\[
\int_{t^\flat-\delta}^t {\mathcal U}_{h_2}^\eps(t,\sigma)
\widehat {B_2 \Pi_1} \widehat {\vec V_1}(\sigma) v_1^\eps(\sigma)d\sigma = 
\int_{t^\flat-\delta}^t {\mathcal U}_{h_2}^\eps(t,\sigma)
\widehat{\Pi_2 B_2\Pi_1}\widehat{\vec V_1}(\sigma) v_1^\eps(\sigma)d\sigma + O(\eps).
\]
By Lemma~\ref{lem:B} and the definition of the eigenvector $\vec V_2$ (see~\eqref{def:V2})
\[
\tfrac12 \Pi_2 B_2 \vec V_1 = \Pi_2(-\partial_t\Pi_2 -\{v,\Pi_2\})\vec V_1 = \gamma\vec V_2.
\] 
According to Proposition~\ref{wpevol1}, we have for the wave packet
\begin{align*}
v^\eps_1(\sigma) &= \U_{h_1}(\sigma,t^\flat) v^\eps_1 (t^\flat) 
=   \U_{h_1}(\sigma,t^\flat) {\rm e}^{iS^\flat /\eps} {\mathcal{WP}}_{z^\flat} \varphi_1(t^\flat) +O(\eps).
\end{align*}
Therefore, 
\begin{align*}
&{1\over 2}  \int_{t^\flat-\delta}^t {\mathcal U}_{h_2}^\eps(t,\sigma)
\widehat {B_2 \Pi_1} \widehat {\vec V_1}(\sigma) v_1^\eps(\sigma)d\sigma\\
&= 
{\rm e}^{iS^\flat /\eps}\, {\mathcal U}_{h_2}^\eps(t,t^\flat)
 \int_{t^\flat-\delta}^t {\mathcal U}_{h_2}^\eps(t^\flat,\sigma) \,\widehat{\gamma\vec V_2}(\sigma)\,
  \U_{h_1}(\sigma,t^\flat) {\mathcal{WP}}_{z^\flat} \varphi_1(t^\flat)d\sigma + O(\eps),
\end{align*}
and, in terms of the function $A_\eps$ is defined in~\eqref{def:Aeps}, we are left at time $t=t^\flat+\delta$ with 
\begin{align*}
 w^\eps_1(t^\flat +\delta) &=O(\eps \delta^{-2} ) + O(\delta\eps^{1/3})  ,\\
w^\eps_2(t^\flat +\delta) &= {\rm e}^{iS^\flat /\eps}\, \U_{h_2}^\eps(t^\flat+\delta,t^\flat)\, A_\eps +O(\eps \delta^{-2} ) + O(\delta\eps^{1/3}).
\end{align*}
\end{proof}

\subsection{Constructing the transfer operator}

Next, we relate the transition term $A_\eps$ to an integral operator that is defined in terms of the crossing parameters $\mu^\flat$ and $(\alpha^\flat,\beta^\flat)$ introduced in Theorem~\ref{theo:WPcodim1}.

\begin{lemma}\label{lem:trans}
Let $k\in\N$. With the assumptions of Proposition~\ref{prop:throughcrossing}, 
there exist 
\begin{itemize}
\setlength{\itemsep}{1ex}
\item[-]
a smooth real-valued map $\sigma\mapsto \Lambda(\sigma)$ with
$\Lambda(0) = 0$, $\dot\Lambda(0) = 0$, $\ddot\Lambda(0) = 2\mu^\flat+ \alpha^\flat\cdot\beta^\flat$,
\item[-] 
a smooth vector-valued map $\sigma\mapsto z(\sigma) = (q(\sigma),p(\sigma))$ with 
$z(0) = 0$, $\dot z(0)= (\alpha^\flat,\beta^\flat)$,
\item[-]
a smooth map $\sigma\mapsto {\mathcal Q}^\eps(\sigma)$ of operators, 
that map Schwartz functions to Schwartz functions, with 
${\mathcal Q}^\eps(0) = \op_1^w(\gamma\vec V_2(t^\flat,z^\flat + \sqrt\eps\bullet))$,
\end{itemize}
such that the transition quantity~$A_\eps$ defined in Lemma~\ref{lem:step1} satisfies
\begin{equation}\label{eq:claim}
A_\eps= \wp^\eps_{z^\flat} {\mathcal T}^\eps \varphi_1(t^\flat) + O(\sqrt\eps \delta)
\end{equation}
in $\Sigma^k_\eps(\R^d)$ for the integral operator ${\mathcal T}^\eps$ defined by
\[
{\mathcal T}^\eps\varphi(y) = 
 \int_{-\delta}^{+\delta} 
{\rm e}^{\frac{i}{\eps} \Lambda(\sigma)} 
{\mathcal Q}^\eps(\sigma) 
{\rm e}^{i p_\eps(\sigma)\cdot (y-q_\eps(\sigma))} \varphi(y-q_\eps(\sigma)) \,d\sigma ,\quad 
\varphi\in{\mathcal S}(\R^d),
\]
where we have used the scaling notation $z_\eps(\sigma) = z(\sigma)/\sqrt\eps$.
\end{lemma} 

\begin{proof}
We use Egorov's semi-classical theorem \cite[Theorem~12]{CombescureRobert} and obtain 
that in $\Sigma^k_\eps(\R^d)$, 
\[
\U_{h_2}^\eps(t^\flat,\sigma) \widehat {\gamma\vec V_2}(\sigma)f = \op^w_\eps( (\gamma\vec V_2)(\sigma)\circ\Phi^{\sigma,t^\flat}_2 )\,\U_{h_2}^\eps(t^\flat,\sigma)f+ O(\eps)
\]
for all $f\in\bigcap_{\ell\ge k}\Sigma^\ell_\eps(\R^d)$. Hence,
\[
A_\eps = 
 \int_{t^\flat-\delta }^{t^\flat+\delta} \op^w_\eps((\gamma\vec V_2)(\sigma)\circ\Phi^{\sigma,t^\flat}_2)\,
\U_{h_2}^\eps(t^\flat, \sigma)\,\U_{h_1}^\eps(\sigma,t^\flat)
\,\wp_{z^\flat}^\eps \varphi_1(t^\flat)d\sigma+O(\delta\eps).
\]
We set 
$
\vec Q_2(\sigma)= (\gamma\vec V_2)(t^\flat+\sigma)\circ\Phi^{\sigma+t^\flat,t^\flat}_2, 
$
and note that $\vec Q_2(0) =  (\gamma\vec V_2)(t^\flat)$. We get after a change of variables 
\begin{align*}
A_\eps &= 
 \int_{-\delta }^{\delta} \widehat{\vec Q_2}(\sigma)\,
\U_{h_2}^\eps(t^\flat, t^\flat+\sigma)\U_{h_1}^\eps(t^\flat+\sigma,t^\flat)
\wp_{z^\flat}^\eps \varphi_1(t^\flat)d\sigma+O(\delta\eps).
\end{align*}
Now we apply successively Proposition \ref{wpevol1} to the evolutions ${\mathcal U}_{h_1}^\eps$ and ${\mathcal U}_{h_2}^\eps$ without encorporating the first amplitude correction, that is, for a basic approximation of 
order $\sqrt\eps$. We obtain
\[
\U_{h_2}^\eps(t^\flat, t^\flat+\sigma)\U_{h_1}^\eps(t^\flat+\sigma,t^\flat)
\wp_{z^\flat}^\eps \varphi_1(t^\flat) = 
{\rm e}^{\frac{i}{\eps}S(\sigma)}
\wp_{\zeta(\sigma)}^\eps {\mathcal M}(\sigma)\varphi_1(t^\flat) + O(\sqrt\eps),
\]
where we  denoted the combined center, phase and metaplectic transform by
\begin{align*}
\zeta(\sigma) &=\Phi_2^{t^\flat,t^\flat+ \sigma}\big(\Phi_1^{t^\flat+\sigma,t^\flat}(z^\flat)\big),\\
S(\sigma) &= 
S_1(t^\flat+\sigma,t^\flat, z^\flat) +S_2(t^\flat,t^\flat+\sigma,\Phi_1^{t^\flat+\sigma, t^\flat}(z^\flat)),\\
{\mathcal M}(\sigma) &= {\mathcal M}[F_2(t^\flat,t^\flat+\sigma,\Phi_1^{t^\flat+\sigma, t^\flat}(z^\flat))] 
\,{\mathcal M}[F_1(t^\flat+\sigma,t^\flat,z^\flat)].
\end{align*}
This implies
\[
A_\eps =   \int_{-\delta}^{+\delta} \widehat{\vec Q_2}(\sigma)
{\rm e}^{\frac{i}{\eps}S(\sigma)}
\wp_{\zeta(\sigma)}^\eps {\mathcal M}(\sigma)\varphi_1(t^\flat) d\sigma+ O(\delta \sqrt \eps). 
\]
We observe that 
\[
\zeta(0) = z^\flat,\quad S(0) = 0, \quad {\mathcal M}(0) = \1,
\]
and write $\zeta(\sigma) = z^\flat + z(\sigma)$ with $z(0) = 0$. By Lemma~\ref{lem:phasis},
\[
\dot z(0) = (\alpha^\flat,\beta^\flat),
\quad
\dot S(0) = p^\flat \cdot\alpha^\flat.
\] 
Moreover, using the group and translation properties of the 
wave packet transform \eqref{comWP} and \eqref{transWP}, we have
\begin{align*}
\wp^\eps_{\zeta(\sigma)} &= {\rm e}^{-\frac{i}{\eps} p^\flat\cdot q(\sigma)}
 \wp^\eps_{z^\flat}\Lambda_\eps^{-1} \wp^\eps_{z(\sigma)}\\*[1ex]
 &=  {\rm e}^{-\frac{i}{\eps} p^\flat\cdot q(\sigma)}
 {\rm e}^{-\frac{i}{2\eps}p(\sigma)\cdot q(\sigma)} \wp^\eps_{z^\flat} 
\Lambda_\eps^{-1} \widehat T^\eps(z(\sigma)) \Lambda_\eps\\*[1ex]
 &=  {\rm e}^{-\frac{i}{\eps} p^\flat\cdot q(\sigma)}
 {\rm e}^{-\frac{i}{2\eps}p(\sigma)\cdot q(\sigma)} \wp^\eps_{z^\flat} 
\widehat T^1(z_\eps(\sigma)), 
\end{align*}
By the translation properties 
of the metaplectic transform \cite[Section~3.3]{CombescureRobert}, we have
\[
\widehat T^1(z_\eps(\sigma)) {\mathcal M}(\sigma)=
 {\mathcal M}(\sigma)\widehat T^1(\widetilde z_\eps(\sigma)) 
\]
with new center
\[
\widetilde z(\sigma) = F_1(t^\flat+\sigma,t^\flat,z^\flat)^{-1} 
F_2(t^\flat,t^\flat+\sigma,\Phi_1^{t^\flat+\sigma, t^\flat}(z^\flat))^{-1} z(\sigma)
\]
We observe that 
\[
\widetilde z(0) = z(0) = 0,\quad \dot{\widetilde z}(0) = \dot z(0) = (\alpha^\flat,\beta^\flat).
\]  
Moreover, in view of the relation~\eqref{w:pseudo}, 
\begin{align*}
\widehat{\vec Q_2}(\sigma)\wp_{\zeta(\sigma)}^\eps {\mathcal M}(\sigma) 
&= 
{\rm e}^{-\frac{i}{\eps} p^\flat\cdot q(\sigma)}
 {\rm e}^{-\frac{i}{2\eps}p(\sigma)\cdot q(\sigma)} \widehat{\vec Q_2}(\sigma) \wp^\eps_{z^\flat} 
 {\mathcal M}(\sigma)\widehat T^1(\widetilde z_\eps(\sigma))\\*[1ex]
 &= 
{\rm e}^{-\frac{i}{\eps} p^\flat\cdot q(\sigma)}
 {\rm e}^{-\frac{i}{2\eps}p(\sigma)\cdot q(\sigma)} 
 \wp^\eps_{z^\flat} \op_1^w(\vec Q_2(\sigma,z^\flat+\sqrt\eps\bullet))
 {\mathcal M}(\sigma)\widehat T^1(\widetilde z_\eps(\sigma)). 
\end{align*}
Since
\[
\widehat T^1(\widetilde z_\eps(\sigma))\varphi_1(t^\flat,y) = 
{\rm e}^{\frac{i}{2}\widetilde q_\eps(\sigma)\cdot\widetilde p_\eps(\sigma)}
{\rm e}^{i \widetilde p_\eps(\sigma)\cdot (y-\widetilde q_\eps(\sigma))} \varphi_1(t^\flat,y-\widetilde q_\eps(\sigma)),
\]
we may introduce the phase $\widetilde\Lambda(\sigma)$ and the operator ${\mathcal Q}^\eps(\sigma)$ acccording to
\begin{align}\nonumber
\widetilde\Lambda(\sigma) &= S(\sigma) - p^\flat\cdot q(\sigma) - p(\sigma)\cdot q(\sigma) + 
\widetilde p(\sigma)\cdot \widetilde q(\sigma),\\\label{def:Qeps(sigma)}
{\mathcal Q}^\eps(\sigma) &= 
\op_1^w(\vec Q_2(\sigma,z^\flat+\sqrt\eps\bullet)) {\mathcal M}(\sigma),
\end{align}
to obtain the approximation
\[
A_\eps =   \wp^\eps_{z^\flat} \int_{-\delta}^{+\delta} 
{\rm e}^{\frac{i}{\eps} \widetilde\Lambda(\sigma)} 
{\mathcal Q}^\eps(\sigma) 
{\rm e}^{i \widetilde p_\eps(\sigma)\cdot (y-\widetilde q_\eps(\sigma))} \varphi_1(t^\flat,y-\widetilde q_\eps(\sigma)) \,d\sigma + O(\delta\sqrt\eps).
\]
We clearly have $\widetilde\Lambda(0) = \dot{\widetilde\Lambda}(0) = 0$ and 
${\mathcal Q}^\eps(0) = \op_1^w((\gamma\vec V_2)(t^\flat,z^\flat+\sqrt\eps\bullet))$, whereas, by Lemma~\ref{lem:phasis},
\[
\ddot{\widetilde\Lambda}(0) = \ddot S(0) - p^\flat\cdot \ddot q(0) = 
2\mu^\flat+ \alpha^\flat\cdot\beta^\flat.
\]
\end{proof}

\begin{remark}\label{generalisation15}
Note that the first step of the proof of Lemma~\ref{lem:trans} can be performed at any order in $\eps$ with a remainder of the form $O(\delta \eps^N)$: pushing the Egorov theorem at higher order, we obtain 
$$A_\eps = 
 \int_{t^\flat-\delta }^{t^\flat+\delta} \widehat{\vec Q^{\eps,N}_2}(\sigma)\,
\U_{h_2}^\eps(t^\flat, \sigma)\U_{h_1}^\eps(\sigma,t^\flat)
\wp_{z^\flat}^\eps \varphi_1(t^\flat)d\sigma+O(\delta\eps^{N+1})$$
$$\mbox{with}\;\;\vec Q_2^{\eps,N}= \vec Q_2 + \eps \vec Q_{2}^{(1)}+\cdots +\eps^N \vec Q_{2}^{(N)}.$$
Similarly, also Proposition~\ref{wpevol1} can be generalized at any order in $\eps$, which then implies
  $$A_\eps=  
 \wp^\eps_{z^\flat} \,{\mathcal T}^{\eps,N} \varphi_1^\eps(t^\flat) +O(\eps^{N/2+1} \delta)$$
 where 
 $\displaystyle{\varphi_1^\eps=\varphi_1+\sqrt\eps \varphi_1^{(1)} + \cdots +\eps^{N/2} \varphi_1^{(N)}}$
 and 
 $${\mathcal T}^{\eps,N}\varphi (y)=\int_{-\delta}^{+\delta}{\rm e}^{\frac{i}{\eps}
\Lambda(\sigma)} {\mathcal Q}^{\eps,N}(\sigma) 
{\rm e}^{
i (y-q_\eps(\sigma)))\cdot p_\eps(\sigma)}
 \varphi(y-q_\eps(\sigma)) d\sigma$$
 for all $\varphi\in{\mathcal S}(\R^d)$. The phase function $\Lambda(\sigma)$ and the phase space center $z(\sigma)$ stay 
 the same as in Lemma~\ref{lem:trans}, while 
 the operator ${\mathcal Q}^{\eps,N}(\sigma)$ is associated with $\vec Q_2^{\eps, N}(\sigma)$ 
 according to~\eqref{def:Qeps(sigma)} by selecting terms up to order $\eps^{N/2}$ in its definition.  
\end{remark}

\subsection{The transfer operator }\label{sec:trsfop}

Consider the family of operators 
\begin{eqnarray*}
{\mathcal T}^\eps\varphi (y)&=& \int_{-\delta}^{+\delta}{\rm e}^{\frac{i}{\eps}
\Lambda(\sigma)}
 {\mathcal Q}^\eps(\sigma) {\rm e}^{ 
i (y-q_\eps(\sigma)))\cdot p_\eps(\sigma)}
 \varphi(y-q_\eps(\sigma)) d\sigma,\quad  \varphi\in{\mathcal S}(\R^d),
\end{eqnarray*}
as introduced in Lemma~\ref{lem:trans}. 
We next describe such an operator~${\mathcal T}^\eps$, when $\eps$ goes to $0$.

\begin{lemma}\label{lem:transap}
Let $k\in\N$.  If $\sqrt\eps \ll\delta\ll 1$, then for all $\varphi\in{\mathcal S}(\R^d)$, 
 \bea\label{transap} 
{\mathcal T}^\eps\varphi = \sqrt\eps \,\mathcal{Q}^\eps(0) \mathcal T^\flat\varphi + O(\sqrt \eps \delta) + O(\eps\delta^{-1})
\eea
in $\Sigma^k_\eps(\R^d)$ with
$\displaystyle{
\mathcal T^\flat = \int_{-\infty}^{+\infty}  {\rm e}^{i\mu^\flat s^2  }  {\rm e}^{is(\beta^\flat\cdot y-\alpha ^\flat \cdot D_y)} \,ds
}$
  \end{lemma}

\begin{proof}
The proof relies on the analysis of the integrand close to $\sigma=0$. 
  We write 
 $$
{\mathcal T}^\eps =\sqrt\eps  \int_{-\delta/\sqrt\eps }^{+\delta/\sqrt\eps }{\rm e}^{\frac{i}{\eps}
\Lambda(\sqrt \eps s) -\frac i 2 q_\eps(s\sqrt\eps) \cdot p_\eps (s\sqrt\eps) } {\mathcal Q}^\eps (s\sqrt\eps ) 
{\rm e}^{i L^\eps (s) }ds 
$$
where
$\displaystyle{L^\eps (s) :=   p_\eps(s\sqrt\eps)\cdot y -p_\eps(s\sqrt\eps) D_y }$
defines a family of self-adjoint operators $s \mapsto L^\eps(s)$ mappping $\mathcal S(\R^d)$ into itself.
Recall that the functions $s\mapsto p_\eps(s\sqrt\eps)$ and $s\mapsto q_\eps(s\sqrt\eps)$ are uniformly bounded with respect to $\eps$, and that $q(0) = p(0) = 0$, while  
\begin{equation}\label{def:mu}
\mu^\flat=\frac12 \left(\ddot\Lambda(0)- \dot q(0) \cdot \dot p(0)\right),\;\;\alpha^\flat=\dot q(0),\;\;\beta^\flat=\dot p(0).
\end{equation}
We set 
$L= \beta^\flat\cdot y-\alpha ^\flat \cdot D_y.$
Using Taylor expansion in $s=0$, we obtain 
$$\Lambda(\sqrt \eps s) -\frac i 2 q_\eps(s\sqrt\eps)\cdot p_\eps(s\sqrt\eps) = \mu^\flat s^2 + \sqrt \eps s^3 f_1(s\sqrt\eps) $$
 with $\sigma \mapsto f_1(\sigma)$ bounded, together with its derivatives, for $\sigma\in [t_0,t_0+T]$. In the following, the notation $f_j$ will denote functions that have the same property. 
 We also have 
$$L^\eps(s) = sL +\sqrt\eps s^2 L^\eps_1(s\sqrt\eps )$$
where the family of operator $\sigma \mapsto L^\eps_1(\sigma)$ maps $\mathcal S(\R^d)$ into itself, for $\sigma\in [t_0,t_0+T]$. Besides, the commutator $[ L, L_1(s\sqrt\eps)]$ is a scalar, and we set
$$\frac 12[ L, L_1(s\sqrt\eps)] =  f_2(s\sqrt\eps)$$
  with the notation we have just introduced. 
  Therefore, by Baker-Campbell-Hausdorff formula
  $${\rm e}^{iL^\eps(s)} = {\rm e}^{isL} {\rm e}^{is^2\sqrt\eps L_1(s\sqrt\eps)}{ \rm e} ^{i  \sqrt\eps s^3 f_2(s\sqrt\eps)}.$$
  Besides,
  $${\rm e}^{i\sqrt\eps s^2 L_1(s\sqrt\eps)} ={\rm Id} + \sqrt\eps s^2 \Theta(s\sqrt\eps)$$
  where the operator-valued map $ \sigma \mapsto \Theta(\sigma)$ is smooth and such that  for all $\sigma\in [t_0,t_0+T]$, the operator $\Theta(\sigma)$ and its derivatives maps $\mathcal S(\R^d)$ into itself.
  Setting $f_3=f_1+f_2$, we deduce that ${\mathcal T}^\eps$ writes 
   \begin{align*}
&{\mathcal T}^\eps=\sqrt\eps  \int_{-\delta/\sqrt\eps }^{+\delta/\sqrt\eps }
{\rm e}^{i\mu^\flat s^2 + \sqrt\eps s^3 f_3(s\sqrt\eps)}  {\mathcal Q}^\eps(s\sqrt\eps ) 
{\rm e}^{i sL  }  ds + R^{\eps,\delta}\\
\mbox{with}\;\;&R^{\eps,\delta}= \eps  \int_{-\delta/\sqrt\eps }^{+\delta/\sqrt\eps }
{\rm e}^{i\mu^\flat s^2 + \sqrt\eps s^3 f_3(s\sqrt\eps)} 
{\mathcal Q} (s\sqrt\eps ) 
{\rm e}^{i sL  }  
s^2 \Theta^\eps(s\sqrt\eps)ds.
\end{align*}
Let us analyze $R^{\eps,\delta}$. For this, we perform an integration by parts. Indeed,  
 $$\partial_s ( \mu^\flat s^2 +\sqrt\eps s^3 f_3(s\sqrt\eps))= 2 \mu^\flat s ( 1+ s\sqrt\eps f_4(s\sqrt\eps))$$
 for some smooth bounded function $f_4$ with bounded derivatives. Moreover, since~$\delta$ is small, 
we have  $1+s\sqrt\eps f_4(s\sqrt\eps) > 1/2$
 for all $s\in]-\delta/\sqrt\eps,+\delta/\sqrt\eps[$. 
 Therefore, we can write 
 \begin{align*}
 R^{\eps,\delta} = &  \left[ \frac {\eps s}{2i\mu^\flat ( 1+ s\sqrt\eps f_4(s\sqrt\eps))}
  {\rm e}^{i\mu^\flat s^2 +i\sqrt\eps s^3 f_3(s\sqrt\eps)}  {\mathcal Q}^\eps (s\sqrt\eps )  {\rm e}^{isL}    \right]_{-\frac\delta{\sqrt\eps}}^{+\frac\delta{\sqrt\eps}} \\
 &\qquad
- \frac{\eps}{2i\mu^\flat} \int_{-\frac{\delta}{\sqrt\eps}}^{+\frac\delta{\sqrt\eps}} 
 {\rm e}^{i\mu^\flat s^2+i\sqrt\eps s^3 f_3(s\sqrt\eps)}
 \frac{d}{ds}\left(\frac{s} {1+ s\sqrt\eps f_4(s\sqrt\eps)}  {\mathcal Q}^\eps (s\sqrt\eps ) \e ^{isL} \right) 
ds,
 \end{align*}
where $\mu^\flat\neq0$ by the transversality condition \eqref{hyp:transversality}.
 We deduce that for all $k\in\N$ and $\varphi\in {\mathcal S}(\R^d)$, we have in $\Sigma^k_\eps(\R^d)$ that
 $R^{\eps,\delta} \varphi= O(\sqrt \eps\delta) + R^{\eps,\delta}_1\varphi$
 with 
 $$R^{\eps,\delta}_1\varphi=
 - \frac{\eps}{2i\mu^\flat} \int_{-\frac{\delta}{\sqrt\eps}}^{+\frac\delta{\sqrt\eps}} 
 {\rm e}^{i\mu^\flat s^2+i\sqrt\eps s^3 f_3(s\sqrt\eps)}
 \left(\frac{s} {1+ s\sqrt\eps f_4(s\sqrt\eps)} {\mathcal Q}^\eps (s\sqrt\eps )  \e ^{isL} L\varphi \right) 
ds.
$$
We then need another integration by parts to obtain that 
$R^{\eps,\delta}_1 \varphi= O(\sqrt \eps\delta) $.  Note that this additional integration by parts is required by the presence of a $s$ without a coefficient $\sqrt\eps$ in  the integrand. We write 
\begin{align*}
R^{\eps,\delta}_1\varphi&=
 - \frac{\eps}{(2i\mu^\flat)^2} \left[
 {\rm e}^{i\mu^\flat s^2+i\sqrt\eps s^3 f_3(s\sqrt\eps)}
 \left(\frac{1} {(1+ s\sqrt\eps f_4(s\sqrt\eps))^2} {\mathcal Q}^\eps (s\sqrt\eps )  \e ^{isL} L\varphi \right) \right]_{-\frac{\delta}{\sqrt\eps}}^{+\frac\delta{\sqrt\eps}} \\
&\;\;\;\; + 
\frac{\eps}{(2i\mu^\flat)^2} \int_{-\frac{\delta}{\sqrt\eps}}^{+\frac\delta{\sqrt\eps}} 
 {\rm e}^{i\mu^\flat s^2+i\sqrt\eps s^3 f_3(s\sqrt\eps)} \frac d {ds}
 \left(\frac{1} {(1+ s\sqrt\eps f_4(s\sqrt\eps))^2} {\mathcal Q}^\eps(s\sqrt\eps )  \e ^{isL} L\varphi \right) 
ds\\
&= O(\delta\sqrt\eps)
\end{align*} 
Therefore, we are left with 
\[
{\mathcal T}^\eps= \sqrt\eps \int_{-\frac\delta{\sqrt\eps}}^{+\frac\delta{\sqrt\eps}}  {\rm e}^{i\mu^\flat s^2  +i\sqrt\eps s^3 f_3(s\sqrt\eps)} {\mathcal Q}^\eps (s\sqrt\eps )  {\rm e}^{isL}  \,ds + O(\sqrt\eps\delta ) .
\]
In the positive part of the integral, we perform the change of variable 
\[
z= s(1+\sqrt\eps s f_3(s\sqrt\eps)/\mu^\flat)^{1/2}
\] 
and observe that 
$s=z(1+\sqrt\eps z g_1(z\sqrt\eps) ) \;\;\mbox{and} \;\; \partial_s z = 1+\sqrt\eps z g_2( z\sqrt\eps)$
for some smooth bounded functions $g_1$ and $g_2$ with bounded derivatives. Note, that here again, we have used that $s\sqrt\eps$ is small in the domain of the integral. 
Besides, there exists a family of operator $\widetilde {\mathcal Q}^\eps (z)$ such that 
${\mathcal Q}^\eps (s\sqrt\eps ) = \widetilde {\mathcal Q}^\eps (z\sqrt\eps)$
with $\widetilde {\mathcal Q}^\eps(0)= {\mathcal Q}^\eps(0)$. 
We deduce that there exists a bounded function of $\delta$ denoted by $b(\delta)$ such that 
\begin{align*}
\mathcal T^\eps = \sqrt\eps \int_{-\frac\delta{\sqrt\eps}}^{+\frac\delta{\sqrt\eps} b(\delta)}  {\rm e}^{i\mu^\flat z^2 } \widetilde {\mathcal Q}^\eps (z\sqrt\eps ) {\rm e}^{iz(1+\sqrt\eps z g_1(z\sqrt\eps))\, L} \frac{dz}{1+\sqrt\eps z g_2( z\sqrt\eps)}.
\end{align*}
A Taylor expansion allows to write 
\begin{align*}
\widetilde {\mathcal Q}^\eps (z\sqrt\eps ) {\rm e}^{iz(1+\sqrt\eps z g_1(z\sqrt\eps))\, L} \frac{1}{1+\sqrt\eps z g_2( z\sqrt\eps)}& =
\widetilde {\mathcal Q}^\eps(0)+ \sqrt\eps  z (\widetilde {\mathcal Q}^\eps_1(z\sqrt\eps)+ z \widetilde {\mathcal Q}^\eps_2(z\sqrt\eps))\\
&= {\mathcal Q}^\eps(0)+ \sqrt\eps  z (\widetilde {\mathcal Q}^\eps_1(z\sqrt\eps)+ z \widetilde {\mathcal Q}^\eps_2(z\sqrt\eps))
\end{align*}
for some smooth operator-valued maps $z\mapsto \widetilde {\mathcal Q}^\eps_j(z\sqrt\eps)$ mapping $\mathcal S(\R^d)$ into itself, such that for all $\varphi\in{\mathcal S}(\R^d)$ the family 
$({\widetilde {\mathcal Q}^\eps_j(z\sqrt\eps)}\varphi)_{\eps>0}$ is bounded in $\Sigma^k_\eps(\R^d)$.
We obtain 
\begin{align*}
&\mathcal T^\eps = 
\sqrt\eps \;  {\mathcal Q}^\eps (0) \int_{-\frac\delta{\sqrt\eps}}^{+\frac\delta{\sqrt\eps} b(\delta)}  {\rm e}^{i\mu^\flat z^2 }  {\rm e}^{izL} dz + \tilde R^{\eps,\delta}\\
\mbox{with}\;\;&
\tilde R^{\eps,\delta} = \eps 
\int_{-\frac\delta{\sqrt\eps}}^{+\frac\delta{\sqrt\eps} b(\delta)}   z \, {\rm e}^{i\mu^\flat z^2 } (\widetilde {\mathcal Q}^\eps_1 (z\sqrt\eps ) +z\widetilde {\mathcal Q}^\eps_2(z\sqrt\eps) )\, dz.
\end{align*}
Arguing by integration by parts as previously, we obtain 
\begin{align*}
\tilde R^{\eps,\delta} =& \eps \left[ \frac{1}{2i\mu^\flat} 
   {\rm e}^{i\mu^\flat z^2 } (\widetilde {\mathcal Q}^\eps_1 (z\sqrt\eps ) +z\widetilde {\mathcal Q}^\eps_2(z\sqrt\eps) )\right]_{-\frac\delta{\sqrt\eps}}^{+\frac\delta{\sqrt\eps} b(\delta)} \\
 &\;\;  -  \frac{\eps}{2i\mu^\flat} 
\int_{-\frac\delta{\sqrt\eps}}^{+\frac\delta{\sqrt\eps} b(\delta)}   \, {\rm e}^{i\mu^\flat z^2 }\frac{d}{dz} (\widetilde {\mathcal Q}^\eps_1 (z\sqrt\eps ) +z\widetilde{\mathcal Q}^\eps_2(z\sqrt\eps) )\, dz
=O(\sqrt\eps \delta) .
\end{align*}
We deduce 
$\displaystyle{
{\mathcal T}^\eps = \sqrt\eps\,  {\mathcal Q}^\eps(0  )\, \int_{-\frac\delta{\sqrt\eps} b(\delta)}^{+\frac\delta{\sqrt\eps}b(\delta)}  {\rm e}^{i\mu^\flat s^2  }  {\rm e}^{isL} \,ds + O(\sqrt\eps\delta ) }$
and it remains to pass to infinity in the domain of the integral. For this, we set $m_\eps=\frac\delta{\sqrt\eps}b(\delta)$ and consider for $\varphi\in{\mathcal S}(\R^d)$,
$$\mathcal G_0^\eps\varphi= \int_{m_\eps}^{+\infty}  {\rm e}^{i\mu^\flat s^2  }   {\rm e}^{isL} \varphi \,ds.$$
We make two successive integration by parts. We write in $\Sigma_k(\R^d)$, 
\begin{align*}
\mathcal G_0^\eps\varphi&= 
 \left[ (2is\mu^\flat )^{-1} {\rm e}^{i\mu^\flat s^2}   {\rm e}^{ isL} \varphi\right]_{m_\eps}^{+\infty} -
\int_{m_\eps }^{+\infty} 
{\rm e}^{i\mu^\flat s^2} 
 \frac{d}{ds} \left(    \frac{{\rm e}^{isL} \varphi}{2is\mu^\flat }\right) ds\\
 &= O( m_\eps^{-1}) \| \varphi\|_{\Sigma^{k}} 
 - \int_{m_\eps }^{+\infty} 
{\rm e}^{i\mu^\flat s^2} \frac{i{\rm e}^{isL} L\varphi}{2is\mu^\flat } ds
 +\int_{m_\eps }^{+\infty} 
{\rm e}^{i\mu^\flat s^2}  \frac{{\rm e}^{isL} \varphi}{2i\mu^\flat s^2 }ds\\
&= O( m_\eps^{-1}) \| \varphi\|_{\Sigma^{k}} - 
 \left[ (2is\mu^\flat )^{-2} {\rm e}^{i\mu^\flat s^2}   i{\rm e}^{ isL} L\varphi\right]_{m_\eps}^{+\infty}
+ \int_{m_\eps }^{+\infty} 
{\rm e}^{i\mu^\flat s^2} \frac{d}{ds}\left(\frac{i{\rm e}^{isL} L\varphi}{(2is\mu^\flat)^2 }\right) ds\\*[1ex]
&= O( m_\eps^{-1}) \left(\| \varphi\|_{\Sigma^{k}}  + \| L\varphi\|_{\Sigma^{k}} + \| L^2\varphi\|_{\Sigma^{k}}\right).  
\end{align*}
We deduce that
$\displaystyle{
{\mathcal T}^\eps = \sqrt\eps\,{\mathcal Q}^\eps(0)\, \int_{-\infty}^{+\infty}  {\rm e}^{i\mu^\flat s^2  } {\rm e}^{isL} \,ds + O(\sqrt\eps\delta ) +O(\eps \delta^{-1}).
}$
  \end{proof}

  \begin{remark} \label{generalisation17}
  Note that the previous remainder terms could again by transformed by integration by parts. This implies that $\mathcal T^\eps \varphi$ has an asymptotic expansion in $\sqrt\eps$ and $\delta$ at any order and each term of the expansion is a Schwartz function. 
  \end{remark}

  \subsection{Proof of Theorem~\ref{theo:WPcodim1} and Corollary~\ref{cor:gaussian}.}\label{subsec:proof}
We now complete the proof of Theorem~\ref{theo:WPcodim1}. We choose 
$\delta=\eps^{2/9}$, and $\eps$ is small enough so that $\eps\leq |t-t^\flat|^{9/2}$. Then, one has $|t-t^\flat|\geq\delta$. If $t\in [t_0, t^\flat-\delta]$, then  Proposition~\ref{prop:propagationt**}  gives the result. If $t\in [t^\flat +\delta, t_0+T]$, then one combines Proposition~\ref{prop:propagationt**} between times $s_1=t^\flat +\delta$ and $s_2=t$ with Proposition~\ref{prop:throughcrossing}. 
 In summary, we obtain an error estimate of order 
$\eps\delta^{-2} = \eps^{1/3}\delta =  \eps^{5/9}$.  

\medskip 

Corollary~\ref{cor:gaussian}  comes from Theorem~\ref{theo:WPcodim1} and point~(3) of Proposition~\ref{prop:T}.


 \appendix

\section{The wave packet transform}\label{appendix:wp}

We discuss here useful properties of the wave-packet transform. We define  the 
Weyl translation operator  $\widehat T^\eps$ 
\[
\widehat T^\eps(z) = {\rm e}^{\frac{i}{\eps}(p\cdot\widehat x-q\cdot\widehat \xi)}, \;\; z=(q,p)\in\R^{2d},
\]
 the semi-classical scaling operator $\Lambda_\eps$ 
\[
\Lambda_\eps\varphi(x) = \eps^{-d/4}\varphi\!\left(\tfrac{x}{\sqrt\eps}\right),
\qquad\varphi\in{\mathcal S}(\R^d),
\]
and we denote by $a_{\eps,z}\in{\mathcal C}^\infty(\R^{2d})$  the function 
 $a_{\eps,z}(w) = a(\sqrt\eps w+z)$, $w\in\R^{2d}$.

 \begin{lemma}\label{lem:coherent} The wave packet transform satisfies for all points $z,z'\in\R^{2d}$ and all smooth functions $a\in{\mathcal C}^\infty(\R^{2d})$
\begin{align}
 \wp^\eps_{z}  &= {\rm e}^{-\frac{i}{2\eps}p\cdot q} \, \hat T^\eps(z)\, \Lambda_\eps,\label{transWP}\\
\wp^\eps_{z+z^\prime}&= {\rm e}^{-\frac{i}{\eps} p\cdot q^\prime}\,\wp^\eps_z\,\Lambda_\eps^{-1} \wp^\eps_{z^\prime},\label{comWP}\\*[0.5em]\label{w:pseudo}
\op_\eps^w(a) \wp^\eps_z &= \wp^\eps_z\, \op_1^w(a_{\eps,z}),
 \end{align}
 \end{lemma}
  
 \begin{proof} We consider $\varphi\in{\mathcal S}(\R^d)$. Then $\widehat T^\eps(z)\varphi$ is the solution 
 at time $t=1$ of the initial value problem
 \[
 i\eps\partial_t\psi = (q\cdot\widehat \xi-p\cdot\widehat x)\psi,\qquad \psi(0) = \varphi.
 \] 
 The explicit form of this solution
 \[
 \psi(t,x) = \e^{-\frac{i}{2\eps}t^2q\cdot p}\, \e^{\frac{i}{\eps}t p\cdot x} \varphi(x-tq)
 \]
 implies for the action of the Weyl translation that
 \[
 \widehat T^\eps(z)\varphi(x) = \e^{-\frac{i}{2\eps}q\cdot p}\, \e^{\frac{i}{\eps}p\cdot x}\varphi(x-q).
 \]
 This yields
 \[
 {\rm e}^{-\frac{i}{2\eps}p\cdot q} \, \hat T^\eps(z)\, \Lambda_\eps\varphi(x) 
 =
 \eps^{-d/4}\,{\rm e}^{-\frac{i}{\eps}p\cdot q} \,\e^{\frac{i}{\eps}p\cdot x}\varphi(\tfrac{x-q}{\sqrt\eps})\\
 =
 \wp^\eps_z\varphi(x). 
\]
 For the commutation property we compute
\begin{align*}
{\rm e}^{-\frac{i}{\eps} p\cdot q^\prime}\wp^\eps_z\Lambda_\eps^{-1} \wp^\eps_{z^\prime}\varphi(x) 
&= 
{\rm e}^{-\frac{i}{\eps} p\cdot q^\prime}\wp^\eps_z
\e^{\frac{i}{\eps}p'\cdot(\sqrt\eps x-q')}\varphi\!\left(\tfrac{\sqrt\eps x-q'}{\sqrt\eps}\right)\\
&= 
{\rm e}^{-\frac{i}{\eps} p\cdot q^\prime}\eps^{-d/4}
\e^{\frac{i}{\eps}p\cdot(x-q)}
\e^{\frac{i}{\eps}p'\cdot(x-q-q')}\varphi\!\left(\tfrac{x-q-q'}{\sqrt\eps}\right)=
\wp^\eps_{z+z^\prime}\varphi(x).
\end{align*}
Moreover,
\begin{align*}
&\wp^\eps_z\,\op_1^w(a_{\eps,z})\varphi(x) \\
&=
\eps^{-d/4} \e^{\frac{i}{\eps}p\cdot(x-q)}(2\pi)^{-d} 
\int_{\R^{2d}} a\!\left(\tfrac{\sqrt\eps}{2}\left(\tfrac{x-q}{\sqrt\eps}+y\right)+q,\sqrt\eps\xi+p\right)) 
\e^{i\xi\cdot((x-q)/\sqrt\eps-y)} \varphi(y) \,dy \,d\xi\\
&=
\eps^{-d/4} \e^{\frac{i}{\eps}p\cdot(x-q)}(2\pi\eps)^{-d} 
\int_{\R^{2d}} a\!\left(\tfrac{1}{2}(x+y')+q,\xi'\right) 
\e^{\frac{i}{\eps}(\xi'-p)\cdot(x-y')} \varphi\!\left(\tfrac{y'-q}{\sqrt\eps}\right) \,dy' \,d\xi'\\
&=
\op_\eps^w(a)\,\wp^\eps_z\varphi(x).
\end{align*}
 \end{proof}

The intertwining property \eqref{w:pseudo}, that relates the wave packet transform with 
Weyl quantization, allows to describe the localisation properties of wave packets as follows. 

 \begin{remark}[Localisation on scale $\sqrt\eps$]\label{rem:apendix}
 Let $\chi\in\mathcal C_0^\infty(\R^{2d})$ be a cut-off function such that $\chi=1$ close to~$0$ and~$0\leq \chi\leq 1$. Define for $R>0$, $\chi_R(z)=\chi(R^{-1} z)$ for all $z\in\R^{2d}$.  Then, for any $k,N\in\N$ and any Schwartz function $\varphi\in{\mathcal S}(\R^d)$
 \[
\left\| \op_1^w(1-\chi_R)\varphi \right\|_{\Sigma^k_1} \le C R^{-N}, 
 \]
 where the constant $C>0$ depends on $k,N$ and the norm of $\varphi$ in $\Sigma^{k+N}_1$. 
Decomposing a wave packet as
 \[
 \wp^\eps_0\varphi = \wp^\eps_0\,\op_1^w(\chi_R)\varphi + \wp^\eps_0\,\op_1^w(1-\chi_R)\varphi,
 \]
 the combination of the above estimate with equation~\eqref{w:pseudo} and the continuity of the wave packet transform as a mapping from $\Sigma^k_1$ to $\Sigma^k_\eps$ yields
 \begin{equation}\label{eq:localisationR}
\left\|\wp^\eps_0\varphi -  \op_\eps^w(\chi_{R\sqrt\eps})\wp^\eps_0\varphi \right\|_{\Sigma^k_\eps} 
\le C R^{-N}. 
 \end{equation}
 \end{remark}

 \section{Algebraic properties of the eigenprojectors}\label{App:projectors}

We  consider a smooth eigenvalue $h(t,z)$ of a matrix-valued Hamiltonian $H(t,z)$, associated with a smooth eigenprojector $\Pi(t,z)$ so that  $H=h\Pi+h^\perp \Pi^\perp$. We emphasize that, in this section,  we just assume smoothness of the projector and make no gap assumption. 
Let us project the solution of the Hamiltonian system \eqref{system} to the eigenspace and consider the function $\tilde w^\eps(t) = 
\widehat\Pi\psi^\eps(t)$. We have
\[
i\eps\partial_t\tilde w^\eps(t) = \left(i\eps\widehat{\partial_t\Pi} + \widehat\Pi\widehat H\right)\psi^\eps(t),
\] 
and by symbolic calculus
\begin{align*}
\widehat\Pi\widehat H &= \widehat{h\Pi} + \frac{\eps}{2i}\widehat{\{\Pi,H\}} + O(\eps^2)= \widehat{h}\widehat{\Pi} -\frac{\eps}{2i}\widehat{\{h,\Pi\}}+ \frac{\eps}{2i}\widehat{\{\Pi,H\}} + O(\eps^2),
\end{align*}
where the order $\eps^2$ remainder will be given a precise meaning in 
Lemma~\ref{lem:preliminary_calcul} and Lemma~\ref{lem:diagonalisation} below. 
Therefore, if we introduce the matrix
\begin{equation}\label{def:B}
B = -2\partial_t \Pi - \{h,\Pi\}+\{\Pi, H\},
\end{equation}
then we may write 
\[
i\eps \partial_t \tilde w^\eps(t) = \widehat h \tilde w^\eps(t) + \frac\eps {2i} \widehat B \psi^\eps(t)
+ O(\eps^2).
\]
Let us examine the algebraic properties of the first order contribution $B$ in more detail.

\begin{lemma}\label{lem:B}
Consider a Hermitian matrix $H = h\Pi+h^\perp\Pi^\perp$ with eigenvalues $h, h^\perp$ and 
corresponding eigenprojectors $\Pi,\Pi^\perp$. 
Then, the matrix $\{\Pi,\Pi\}$ is skew-symmetric and diagonal,
\[
\Pi^\perp\{\Pi,\Pi\}\Pi = \Pi^\perp\{\Pi,\Pi\}\Pi = 0.
\]
The matrix $B$ defined in \eqref{def:B} satisfies 
\[
B\Pi = -2(\Omega+K) = 2i\Theta\Pi\;\;
\mbox{and}\;\;
\Pi^\perp B\Pi^\perp = (h-h^\perp)\Pi^\perp\{\Pi,\Pi\}\Pi^\perp,
\]
where the matrices $\Omega$, $K$, and $\Theta$ have been introduced in \eqref{def:Omega}, \eqref{def:K}, 
and \eqref{def:Theta}. Moreover, the matrix $\Omega$ is skew-symmetric and $\Theta$ self-adjoint.
\end{lemma}

\begin{proof}
We use the relation 
 $\{{\mathcal A},{\mathcal B}{\mathcal C}\} -\{{\mathcal A}{\mathcal B},{\mathcal C}\} =
 \{{\mathcal A},{\mathcal B}\} {\mathcal C}-{\mathcal A}\{{\mathcal B},{\mathcal C}\}$.
  and apply it to ${\mathcal A}={\mathcal B}={\mathcal C}=\Pi$. Since $\Pi^2 = \Pi$, we obtain 
 $0=\{\Pi,\Pi\}\Pi-\Pi\{\Pi,\Pi\}$
and therefore
 \[
 \Pi^\perp\{\Pi,\Pi\} \Pi =\Pi\{\Pi,\Pi\} \Pi^\perp =0.
 \]
Besides, by the definition of the Poisson bracket, we have
$\{\Pi,\Pi\}^* = -\{\Pi,\Pi\}$, so that $\{\Pi,\Pi\}$ and $\Omega = -\frac12(h-h^\perp)\Pi\{\Pi,\Pi\}\Pi$ are skew-symmetric. 
 In view of 
\begin{align*}
\{\Pi,H\} &= (h-h^\perp)\{\Pi,\Pi\} - \{ h, \Pi\}\Pi -\{h^\perp,\Pi\}\Pi^\perp,\\
\{ h,\Pi\} &= \{h,\Pi\} \Pi+ \{ h,\Pi\} \Pi^\perp,
\end{align*}
we obtain that
\begin{align*}
B &= -2\partial_t\Pi-\{ h,\Pi\} +\{\Pi,H\}\\
&= -2\partial_t\Pi + (h-h^\perp)\{\Pi,\Pi\} - 2 \{ h,\Pi\} \Pi - \{h+h^\perp,\Pi\}\Pi^\perp.
\end{align*}
Hence,
\begin{align*}
B\Pi &= -2\Pi^\perp(\partial_t\Pi+\{h,\Pi\})\Pi + (h-h^\perp)\Pi\{\Pi,\Pi\}\Pi
= -2(K+\Omega)
 \end{align*}
 \[
 \mbox{and}\;\;
 \Pi^\perp B\Pi^\perp = (h-h^\perp)\Pi^\perp\{\Pi,\Pi\}\Pi^\perp.
 \]
 The matrix $\Theta = i\Omega+i(K-K^*)$ is hermitian, since 
 $\Theta^* = -i\Omega^* - i(K^*-K) = \Theta$. It also satisfies 
 $2i\Theta\Pi = 2i(i\Omega + iK) \Pi= B\Pi$.
 \end{proof}

Decomposing the matrix $B = B\Pi+B\Pi^\perp$, we may  
view the contribution associated with the projector $\Pi$ as an effective dynamical correction to the eigenvalue $h$. 
We obtain the following:

\begin{lemma}\label{lem:preliminary_calcul}
Let $H = h\Pi+h^\perp\Pi^\perp$ be a smooth matrix-valued Hamiltonian with smooth eigenvalues 
$h,h^\perp$ and smooth eigenprojectors $\Pi,\Pi^\perp$. Then, there exists a smooth matrix-valued symbol 
$R^\eps$ such that 
\begin{align}\label{preliminary_calcul}
\widehat \Pi (i\eps \partial_t -\widehat H) =& (i\eps\partial_t - \widehat h -\eps \widehat \Theta)\widehat \Pi 
+\frac{\eps}{2i} \widehat {B\Pi^\perp} \widehat \Pi^\perp +\eps^2 \widehat R^\eps,
\end{align}
where the matrices $B$ and $\Theta$ have been defined in \eqref{def:B} and \eqref{def:Theta}, respectively. 
If the Hamiltonian and its eigenvalues are of subquadratic growth \eqref{hyp:H}, while the projectors grow at most polynomially \eqref{bound:projector}, then for all $k\in\N$ there exist $C_k>0$ and $\ell\in\N$ such that 
\[
\sup_{t\in[t_0,t_0+T]}\|\widehat R^\eps(t)\varphi\|_{\Sigma^k_\eps} \le C_k \|\varphi\|_{\Sigma^\ell_\eps}\;\;
\forall\varphi\in\Sigma^k_\eps(\R^d).
\]
\end{lemma}

\begin{proof}
We  write
\begin{align*}
\widehat \Pi (i\eps \partial_t -\widehat H) =& (i\eps \partial_t -\widehat h)\widehat \Pi -i\eps \widehat{\partial_t\Pi} 
+ \widehat h \widehat \Pi - \widehat \Pi\widehat H.
\end{align*}
The symbolic calculus gives
\begin{align*}
\widehat h \widehat \Pi - \widehat \Pi\widehat H= \frac \eps{2i} (\{ h,\Pi\} -\{\Pi,H\}) + \eps^2 R^\eps,
\end{align*}
where the remainder $\widehat R^\eps(t)$ satisfies the claimed estimate due to the growth assumptions on 
the symbols $h,H$ and $\Pi$.  
In view of Lemma~\ref{lem:B}, we have
\begin{align*}
-i\partial_t \Pi+ \frac1{2i}\left( \{ h,\Pi\} -\{\Pi,H\}\right) = \Theta\Pi + \frac{1}{2i} B\Pi^\perp,
\end{align*}
which concludes our proof.
\end{proof}

We note that for the projected solution $\tilde w^\eps(t) = \widehat\Pi \psi^\eps(t)$, equation \eqref{preliminary_calcul} implies an evolution equation of the form
\[
i\eps\partial_t \tilde w^\eps(t) = (\widehat h + \eps\widehat\Theta)\tilde w^\eps(t) - 
\frac{\eps}{2i} \widehat{B\Pi^\perp} \widehat\Pi^\perp \psi^\eps(t) + O(\eps^2).
\]
In a next step we use the matrix $B$ for introducing the first order super-adiabatic correction of the 
eigenprojector $\Pi$, following ideas from~\cite{sjo_ad,emwe,bi,MS,N1,N2,Te}. 

\begin{definition}\label{def:superadiabatic}
We assume that $H$ is a smooth Hermitian matrix that has two smooth eigenvalues $h$ and $h^\perp$ 
and smooth eigenprojectors $\Pi$ and $\Pi^\perp$, that is, $H = h\Pi + h^\perp\Pi^\perp$. 
The first super-adiabatic corrector of $\Pi$ is  the hermitian matrix $\mathbb P = \mathbb P^*$ defined by 
\begin{align*}
\Pi\mathbb P \Pi^\perp 
& = \frac{i}{h-h^\perp}\Pi\left(\partial_t \Pi +\frac12 \{h+h^\perp,\Pi\}\right)\Pi^\perp,\\ 
\Pi^\perp\mathbb P \Pi 
& = -\frac{i}{h-h^\perp}\Pi^\perp\left(\partial_t \Pi +\frac12 \{h+h^\perp,\Pi\}\right)\Pi,\\
\Pi\mathbb P\Pi &= -\frac{1}{2i}\Pi\{\Pi,\Pi\}\Pi
,\;\;\;\;\;\;
\Pi^\perp \mathbb P\Pi^\perp = \frac{1}{2i}\Pi^\perp\{\Pi,\Pi\}\Pi^\perp.
\end{align*}
\end{definition}
Note that one has $\displaystyle{
\mathbb P \Pi^\perp = \frac{1}{2i}(h-h^\perp)^{-1} B\Pi^\perp \;\;\mbox{and}\;\;
\mathbb P\Pi = -\frac{1}{2i}(h-h^\perp)^{-1} B\Pi}$.

Note that the diagonal part of the matrix $\mathbb P$ is smooth, while the off-diagonal part of~$\mathbb P$ is singular on the crossing set~$\Upsilon = \{f = 0\}$. Besides, for all $\beta\in\N^{2d}$ and $R>0$,
 \begin{equation}\label{est:corrector} 
\exists C_{\beta,R}>0,\;\; \forall z\in B(0,R) \cap \{f(t,z)>\delta\},\;\;\forall t\in\R,\;\; \| \partial^\beta _z\mathbb P(t,z)\| \leq C_{\beta,R} \,\delta^{|\beta|+1}.
\end{equation}
The main interest in the corrector $\mathbb P$ comes from the following relations: 

\begin{lemma}\label{lem:super}
With the assumptions of Definition~\ref{def:superadiabatic},
the corrector 
matrix $\mathbb P$ satisfies 
\begin{equation}\label{eq:subquad}
[H,{\mathbb P}]= i\partial_t \Pi -\frac{1}{2i}
(\{H,\Pi\}-\{\Pi,H\} )\quad\mbox{and}\quad
\mathbb P \Pi+\Pi \mathbb P =\mathbb P -  \frac 1 {2i} \{\Pi,\Pi\},
\end{equation}
as well as
\[
i\partial_t\Pi + \mathbb P(H-h) + \frac 1{2i} \{\Pi, H+h\} = \Theta\Pi\quad\mbox{and}\quad i(\partial_t\Pi + \{h,\Pi\}) = [\Theta,\Pi],
\]
where the matrix $\Theta$ is given by~\eqref{def:Theta}.
\end{lemma}

\begin{proof} Since $H$ is acting as a scalar on ${\rm Ran}\Pi$ and ${\rm Ran}\Pi^\perp$, we have
\begin{align*}
[H,\mathbb P] &= [H,\Pi\mathbb P\Pi^\perp + \Pi^\perp \mathbb P \Pi] = 
(h-h^\perp)\Pi\mathbb P\Pi^\perp + (h^\perp-h)\Pi^\perp\mathbb P \Pi\\
&= i(\partial_t \Pi + \tfrac12 \{h+h^\perp,\Pi\}).
\end{align*}
Since
$$\{H,\Pi\} - \{\Pi,H\} = \Pi\{h,\Pi\} + \Pi^\perp\{h^\perp,\Pi\} - \{\Pi,h\}\Pi - \{\Pi,h^\perp\}\Pi^\perp= \{h+h^\perp,\Pi\},$$
we have proven the first equation. For the second equation, we calculate
$$\mathbb P\Pi + \Pi\mathbb P = 2\Pi \mathbb P \Pi + \Pi^\perp\mathbb P \Pi + \Pi \mathbb P \Pi^\perp 
= \mathbb P + \Pi\mathbb P \Pi - \Pi^\perp \mathbb P \Pi^\perp \\
= \mathbb P - \tfrac{1}{2i}\{\Pi,\Pi\},$$
where we have used that $\{\Pi,\Pi\}$ is diagonal. For the first relation with $\Theta$, we write $H-h = (h^\perp-h)\Pi^\perp$ and obtain
\[
\mathbb P(H-h) = (h^\perp-h)\mathbb P\Pi^\perp = -\frac{1}{2i}B\Pi^\perp.
\]
Therefore, by Lemma~\ref{lem:B},
\[
i\partial_t\Pi + \mathbb P(H-h) + \frac 1{2i} \{\Pi, H+h\} = \frac{1}{2i}B - \frac{1}{2i}B\Pi^\perp = \Theta\Pi.
\]
For the commutator of $\Theta$ and $\Pi$, we have
\begin{align*}
[\Theta,\Pi] &= i[\Omega,\Pi]+ i[K,\Pi] - i[K^*,\Pi] \\*[1ex]
&= i\Pi^\perp(\partial_t\Pi + \{h,\Pi\})\Pi + i\Pi^\perp(\partial_t\Pi + \{h,\Pi\})\Pi =i(\partial_t\Pi + \{h,\Pi\}).
\end{align*}
\end{proof}

If the crossing set $\Upsilon$ were empty and all the symbols in consideration were bounded, the relations of Lemma~\ref{lem:super} would imply that setting 
$\Pi^\eps= \Pi+ \eps \mathbb P,$
then $\widehat \Pi^\eps$ would be ``better''  than $\widehat \Pi$ in terms of being an eigenprojector 
of $\widehat H$: in $\mathcal L(L^2(\R^d))$,  
$$\widehat {\Pi^\eps} \widehat {\Pi^\eps } =\widehat {\Pi^\eps} +O(\eps^2)
\quad\text{and}\quad
\widehat {\Pi^\eps}(-i\eps\partial_t+\widehat H) =  
(-i\eps\partial_t+\widehat h+\eps \widehat \Theta)\widehat {\Pi^\eps} +O(\eps^2),$$
while the estimate would be only $O(\eps)$ when using the uncorrected $\widehat \Pi$. However, because the symbols we consider  are smooth only outside $\Upsilon$, we need to use cut-off functions to correctly state  such properties.

\begin{lemma}\label{lem:diagonalisation}
Let $I$ be an interval of $\R$  and  $\chi^\delta ,\tilde \chi^\delta\in\mathcal C(I,{\mathcal C}_0^\infty(\R^{2d}))$ be two cut-off functions that satisfy:
 \begin{enumerate}
 \item For any $t\in I$ and $z$ in the support of $\chi^\delta(t)$ and $\tilde \chi^\delta(t)$ we have
$|(h-h^\perp)(t,z)|>\delta$.
\item The functions $\chi^\delta$ and $\tilde\chi^\delta$ 
satisfy 
$$\partial_t \chi^\delta + \left\{h, \chi^\delta \right\}=0,\;\;\partial_t \tilde\chi^\delta + \left\{h, \tilde\chi^\delta \right\}=0.$$
\item The functions $\tilde\chi^\delta$ are supported in $\{\chi^\delta=1\}$. 
\end{enumerate}
Let $k\in\N$. Then, we have for all $t\in I$ in $\Sigma^k_\eps$,
\begin{align*}
\label{gamma}
 \widehat {\tilde \chi^\delta}\left(-i\eps\partial_t +(\hat h + \eps \widehat \Theta)\right) \widehat{\chi^\delta\Pi^\eps}&  = \widehat {\tilde \chi^\delta}\widehat {\chi^\delta \Pi^\eps}(-i\eps\partial_t + \widehat H )+ O(\eps^2 \delta^{-2}).
\end{align*}
In particular, the function $w^\eps(t) = \widehat {\tilde \chi^\delta}\widehat {\chi^\delta \Pi^\eps}\psi^\eps(t)$ satisfies for all $t\in I$ in $\Sigma^k_\eps$,
$$i\eps \partial_t w^\eps(t)= (\hat h + \eps \widehat \Theta) w^\eps(t) +O(\eps^2\delta^{-2}).$$
\end{lemma}

\begin{proof}
We write 
\begin{align*}
\widehat {\chi^\delta \Pi^\eps}(-i\eps\partial_t + \widehat H) &= {\rm op}_\eps ( \chi^\delta h \Pi ) + \eps {\rm op}_\eps ( \chi^\delta \mathbb P H + \frac 1{2i} \{ \chi^\delta \Pi, H\}) - \widehat {\chi^\delta \Pi^\eps}(i\eps\partial_t) + \eps^2 \widehat{R^\delta(t)},
\end{align*}
where the remainder $R^\delta(t)$ depends on first order derivatives of $\chi^\delta\mathbb P$ and $H$ as 
well as second order derivatives of $\chi^\delta\Pi$ and $H$. Hence, $\widehat R^\delta(t) = O(\delta^{-2})$.
Next, we write
\begin{align*}
{\rm op}_\eps ( \chi^\delta h \Pi ) - \widehat {\chi^\delta \Pi^\eps}(i\eps\partial_t) &= 
{\rm op}_\eps (h) {\rm op}_\eps( \chi^\delta\Pi) -\frac{\eps}{2i}  {\rm op}_\eps(\{h, \chi^\delta\Pi\})\\
&\qquad-(i\eps\partial_t) \widehat {\chi^\delta \Pi^\eps}+ i\eps{\rm op}_\eps ( \partial_t(\chi^\delta \Pi^\eps)) + 
\eps ^2\widehat{\rho_2^\delta(t)},
\end{align*}
where $\rho_2^\delta(t)$ depends linearly on second derivatives of $\chi^\delta\Pi$ and $h$. 
By Lemma~\ref{lem:super}, one part of the first order contributions can be combined according to 
\begin{align*}
i\partial_t\Pi + \mathbb P(H-h) + \frac 1{2i} \{\Pi, H+h\} = \Theta\Pi.
\end{align*}
All this implies
\begin{align*}
\widehat {\chi^\delta \Pi^\eps}(-i\eps\partial_t + \widehat H) &=(- i\eps \partial_t +  {\rm op}_\eps (h +\eps \Theta )) {\rm op}_\eps( \chi^\delta\Pi ^\eps ) \\
&\qquad+ \eps {\rm op}_\eps (\rho_1^\delta(t)) + 
\eps^2\op_\eps(\rho_2^\delta(t) + R^\delta(t)),
\end{align*}
where the remainder is given by
\begin{align*}
\rho_2^\delta(t) &= \Pi (i\partial_t\chi^\delta + \frac{1}{2i}\{\chi^\delta,H+h\})\\
&= \Pi (i\partial_t\chi^\delta + \frac{1}{i}\{\chi^\delta,h\} + \frac{1}{2i}(h-h^\perp)\{\chi^\delta,\Pi\} )\\
&= \frac{1}{2i}(h-h^\perp)\Pi\{\chi^\delta,\Pi\}\Pi^\perp,
\end{align*}
since $\partial_t\chi^\delta+\{h,\chi^\delta\}=0$. We note that $\rho_1^\delta(t)$ and $\rho_2^\delta(t)$ are smooth symbols, depending linearly on derivatives of $\chi^\delta$ and thus are~$0$ on the support of $\tilde\chi^\delta$. The latter observation implies the first result.
For the function~$w^\eps(t)$ we then have
\begin{align*} 
&(i\eps \partial_t -\hat h - \eps \widehat \Theta)w^\eps(t)\\
&=
\widehat {\tilde \chi^\delta}(i\eps \partial_t -\hat h - \eps \widehat \Theta)\widehat {\chi^\delta \Pi^\eps}\psi^\eps(t)
+ [ (i\eps \partial_t -\hat h - \eps \widehat \Theta), \widehat {\tilde \chi^\delta}]\widehat {\chi^\delta \Pi^\eps}
\psi^\eps(t)\\
&= \widehat {\tilde \chi^\delta}\widehat {\chi^\delta \Pi^\eps}(i\eps\partial_t -\widehat H) \psi^\eps(t) 
+  [ (i\eps \partial_t -\hat h - \eps \widehat \Theta), \widehat {\tilde \chi^\delta}] \widehat {\chi^\delta \Pi^\eps}
\psi^\eps(t)+\eps^2\widehat{R^\delta(t)}.
\end{align*}
Moreover, since $\partial_t \tilde \chi^\delta + \{ h,\tilde\chi^\delta\}=0$, we have  
$$[ (i\eps \partial_t -\hat h - \eps \widehat \Theta), \widehat {\tilde \chi^\delta}]= 
\eps^3 \widehat{r^\delta_3(t)} + \eps^2 \widehat{r^\delta_1(t)},$$
where the first part of the remainder $r^\delta_3(t)$ depends on third 
derivatives of $h$ and $\tilde\chi^\delta$, while the second part $r^\delta_1(t)$ depends on first derivatives of 
$\Theta$ and $\tilde\chi^\delta$. Since $\delta\gg\sqrt\eps$, we have
\[
[(i\eps \partial_t -\hat h - \eps \widehat \Theta), \widehat {\tilde \chi^\delta}] = O(\eps^2\delta^{-1}).
\]
Using $(i\eps\partial_t -\widehat H) \psi^\eps(t)=0$, we obtain the equation for $w^\eps(t)$.
\end{proof}

\section{Parallel transport}\label{app:parallel}
We prove here Proposition~\ref{prop:eigenvector} that provides the time-dependent eigenvector $\vec V(t,z)$ defined by parallel transport. We adapt the proof of \cite[Proposition~C.1]{CF11} to account for the matrix 
$\Omega(t,z)$, noting that we only require that $\Omega(t,z)$ is a skew-symmetric matrix mapping into the range of $\Pi(t,z)$. 

\begin{proof}
We consider the solution $\vec V(t,z)$ of the parallel transport equation and 
set $Y(t,z) = \vec V(t,\Phi_h^{t,t_0}(z))$. We observe that $Y(t,z)$ solves the equation
\begin{eqnarray}
\nonumber 
\partial_t Y(t,z)&=&\partial_t \vec V(t,\Phi_h^{t,t_0}(z)) +J\partial_zh(\Phi_h^{t,t_0}(z)) V(t,\Phi_h^{t,t_0}(z))\\
\label{eq:Y}
&=&\Omega(t,\Phi^{t,t_0}_h(z))Y(t,z) + K(t,\Phi^{t,t_0}_h(z))Y(t,z).
\end{eqnarray}
In particular, since $\Omega(t,z)$ maps into the range of $\Pi(t,z)$, 
\[
\Pi^\perp(t,\Phi_h^{t,t_0}(z))\  \partial_t Y(t,z) = K(t,\Phi^{t,t_0}_h(z))Y(t,z).
\]
We now start proving that for $z\in U$,
$\Pi(t,\Phi_h^{t,t_0}(z))Y(t,z)=Y(t,z)$, 
or equivalently that 
$$Z(t,z)=\Pi^\perp(t,\Phi_h^{t,t_0}(z))Y(t,z)$$
is constant and equal to $0$.  We compute 
\[
\partial_t Z(t,z) 
=\left( -\partial_t\Pi(t,\Phi_h^{t,t_0}(z))- J\partial_zh(\Phi^{t,t_0}_h(z))\partial_z\Pi(t,\Phi^{t,t_0}_h(z)) 
+ K(t,\Phi^{t,t_0}_h(z))\right) Y(t,z).
\]
We recall that $K = (\1-\Pi)(\partial_t\Pi + \{h,\Pi\})\Pi$. Since all derivatives of the projector are off-diagonal, we have
\[
-\partial_t\Pi - \{h,\Pi\} + K = - \Pi\left(\partial_t\Pi + \{h,\Pi\}\right)\Pi^\perp
\]
and therefore
\[
\partial_t Z(t,z) =- \Pi(t,\Phi_h^{t,t_0}(z))
\left(\partial_t\Pi(t,\Phi_h^{t,t_0}(z)) + J\partial_zh(\Phi^{t,t_0}_h(z))\partial_z\Pi(t,\Phi^{t,t_0}_h(z)) \right) Z(t,z).
\]
In particular, $\partial_t Z(t,z)$ is an element of the range of $\Pi(t,\Phi^{t,t_0}_h(z))$ and thus 
orthogonal to $Z(t,z)$. Hence, its norm is constant,
$Z(t,z)=0$  and $Y(t,z)\in {\rm Ran}\, \Pi(t,\Phi^{t,t_0}_h(z))$.  

\medskip 

Besides, we have for any $z\in\R^{2d}$
\[
\partial_t Y(t,z) \cdot Y(t,z) = \Omega(t,\Phi^{t,t_0}_h(z))Y(t,z)\cdot Y(t,z) + 
K(t,\Phi^{t,t_0}_h(z))Y(t,z)\cdot Y(t,z) = 0,
\]
because 
\[
\Omega(t,z)^*=-\Omega(t,z)\ \text{and}\ K(t,z)=\Pi^\perp(t,z)K(t,z).
\] 
Therefore, $\|Y(t,z)\|_{\C^N}=1$. 
\end{proof}

\begin{remark}[Polynomial growth of the eigenvector]\label{rem:poly_eigenvec} The above proof shows that the time-evolution of 
$Y(t,z) = \vec V(t,\Phi^{t,t_0}_h(z))$  
is generated by a norm-conserving evolution operator, that is, 
$Y(t) = {\mathcal L}(t,t_0) Y(t_0).$
This observation allows to literally repeat the inductive argument in the proof of \cite[Proposition~C.1]{CF11} for 
inferring from a polynomial bound on the projector $\Pi(t,z)$ a polynomial 
bound for the eigenvector $\vec V(t,z)$. Indeed, if \eqref{bound:projector} holds for $\Pi(t,z)$, then for all   
$T>0$ and $\beta\in\N_0^{2d+1}$ there exists a constant $c_{\beta,T}>0$ such that
\[
\sup_{t\in[t_0,t_0+T],|z|\ge r_0}\|\partial_{t,z}^{\beta} \vec V(t,z)\| \le c_{\beta,T} \langle z\rangle^{|\beta|(1+n_0)}.
\]
\end{remark}

\section{The phase $\Lambda(\sigma)$ and the function $\zeta(\sigma)$}\label{appendix:phase}

 \begin{lemma}\label{lem:phasis}
 Let $\Lambda$ and $\zeta$ be defined as 
 \begin{align}\nonumber
\zeta(\sigma) &=\Phi_2^{t^\flat,t^\flat+ \sigma}\big(\Phi_1^{t^\flat+\sigma,t^\flat}(z^\flat)\big),\\\label{Lambda:eq1}
\Lambda(\sigma) &= S_1(t^\flat+\sigma,t^\flat, z^\flat) +S_2(t^\flat,t^\flat+\sigma,\Phi_1^{t^\flat+\sigma, t^\flat}(z^\flat)) - q(\sigma)\cdot p^\flat.
\end{align}
 We have 
 \begin{align}
 \label{zeta(0)}
 \zeta(0)&=(q(0),p(0))=z^\flat,\;\;\dot\zeta(0)= (\dot q(0),\dot p(0))=J\partial_z (h_1-h_2) (t^\flat, z^\flat)\\
 \label{lambda(0)}
  \Lambda(0) &=\dot\Lambda(0)= 0,\\
 \label{sd}
\ddot\Lambda (0)&= \partial_t (h_2-h_1)-\partial_q h_2\cdot \partial_p(h_2-h_1)+\partial_p h_1 \cdot \partial_q (h_2-h_1)
\end{align}
 \end{lemma}
 
In particular, we have 
\begin{align*}
&\frac12(\ddot\Lambda(0) -{\dot p(0)\cdot \dot q(0)})\\
&= \frac{1}{2}\left( \partial_t(h_2-h_1) - \partial_q h_2\cdot \partial_p (h_2-h_1)+\partial_p h_1\cdot \partial_q(h_2-h_1) 
+\partial_p(h_2-h_1) \cdot\partial_q (h_2-h_1)\right)\\
&=  \frac{1}{2}\left( \partial_t(h_2-h_1) - \partial_q h_1\cdot \partial_p (h_2-h_1)+\partial_p h_1\cdot \partial_q(h_2-h_1)\right) \\
&=  \frac{1}{2}\left( \partial_t(h_2-h_1) +\left\{\frac{h_1+h_2}{2}, h_2-h_1\right\}\right)
\end{align*}
which yields that \eqref{def:mu} is consistent with \eqref{def:lambda}.

 \begin{proof}
 We begin with the function $\zeta$ and 
 we compute the Taylor expansion at the order~2 for $(q(\sigma),p(\sigma))=\zeta(\sigma)-z^\flat $ at $\sigma=0$.
Let  be $h=h_1, h_2$.  We have :
 \begin{align}\label {DTPhi}
 \Phi_h^{t,t_0}(z) = z &+(t-t_0)J\partial_z h(t_0,z) + \frac{(t-t_0)^2}{2}\left(J\partial^2_{t,z}h(t_0,z)+ J\partial^2_{z,z}h(t_0,z)J\partial_zh(t_0,z)\right)
 \\
 \nonumber
 & + O(|t-t_0|^3).
\end{align} 
Applying this formula, we obtain (omitting the argument $(t^\flat, z^\flat)$ in the functions $h_1$, $h_2$ and their derivatives)
\begin{align*} 
&\Phi_1^{t^\flat+\sigma,t^\flat}(z^\flat)= z^\flat+\sigma J\partial_z h_1 +\frac{\sigma^2}2 \left(J\partial^2_{t,z}h_1+ J\partial^2_{z,z}h_1J\partial_zh_1\right)
  + O(|\sigma|^3),\\
&  \zeta(t) = \Phi_1^{t^\flat+\sigma,t^\flat}(z^\flat) -\sigma J\partial_z h_2(t^\flat+\sigma,\Phi_1^{t^\flat+\sigma,t^\flat}(z^\flat)) 
  +\frac{\sigma^2}2 \left(J\partial^2_{t,z}h_2+ J\partial^2_{z,z}h_2J\partial_zh_2\right)
  + O(|\sigma|^3).
  \end{align*}
  We deduce 
  \begin{align*} 
  \zeta(t) &= z^\flat +\sigma J\partial_z (h_1-h_2) + O(|\sigma|^3)\\
  &+\frac{\sigma^2}2 \left(J\partial^2_{t,z}(h_1-h_2)+ J\partial^2_{z,z}(h_1-h_2)J\partial_zh_1
  +J\partial^2_{z,z} h_2 J\partial_z (h_2-h_1)\right),
  \end{align*} 
and, for further use, the relation 
\begin{align}
\label{C5}
-p^\flat \dot q(0) =& -p^\flat \cdot\partial_q(h_1-h_2),\\
\label{C5''}
-p^\flat \cdot \ddot q(0)=&-p^\flat \cdot (\partial^2_{t,p}(h_1-h_2)+ \partial^2_{z,p}(h_1-h_2)J\partial_zh_1
  +\partial^2_{z,p} h_2 J\partial_z (h_2-h_1))
\end{align}

\medskip

 We continue with the function $\Lambda$ (defined in~\eqref{Lambda:eq1}) and we use Taylor expansion of the actions for general Hamiltonian $h$. In view of~\eqref{def:S} and~\eqref{DTPhi}, we have (omitting the argument $(t_0,z_0)$ in the terms of the form $\partial^\alpha h(t_0,z_0)$)
\begin{align*}
S(t,t_0,z_0) &=  \int_{t_0}^t (p_0-(s-t_0)\partial_qh)\cdot (\partial_p h+(s-t_0) (\partial^2_{t,p}h +\partial^2_{z,p}hJ\partial_z h) )ds\\
&\qquad
-\int_{t_0}^{t} (h+(s-t_0)\partial_t h ) ds  +O((t-t_0)^3)\\
&= (p_0\cdot \partial_p h-h)(t-t_0) -\frac{(t-t_0)^2}{2} (\partial_t h+\partial_qh \cdot \partial_p h -p_0\cdot (\partial_{t,p}^2 h+ \partial^2_{z,p}hJ\partial_z h))\\
&\qquad
+O((t-t_0) ^3).
\end{align*} 
We first apply the formula with $h=h_1$, $t=t^\flat+\sigma$, $t=t^\flat$ and $z=z^\flat$, which gives (when the arguments of the functions are omitted, they are fixed to $(t^\flat,z^\flat)$)
\begin{align*}
S_1(t^\flat+\sigma,t^\flat,z^\flat) = &\, \sigma(p\cdot \partial_p h_1-h_1)\\
&-\frac{\sigma^2}{2} (\partial_t h_1+\partial_qh_1 \cdot \partial_p h_1 -p\cdot (\partial_{t,p}^2 h_1+ \partial^2_{z,p}h_1J\partial_z h_1))
+O(\sigma ^3).
\end{align*}
We now use the same formula with $h=h_2$, $t=t^\flat$, $t_0=t^\flat+\sigma$, $z_0=\Phi_1^{t^\flat+\sigma, t^\flat}(z^\flat)$. We obtain
\begin{align*}
S_2(t^\flat,t^\flat+\sigma,\Phi_1^{t^\flat+\sigma,t^\flat}(z^\flat)) &=\\
- \sigma &(p_1(t^\flat+\sigma,t^\flat,z^\flat)\cdot \partial_p h_2(t^\flat+\sigma,\Phi_1^{t^\flat+\sigma,t^\flat}(z^\flat)) -h_2(t^\flat+\sigma,\Phi_1^{t^\flat+\sigma,t^\flat}(z^\flat)) )\\
& -\frac{\sigma^2}{2} (\partial_t h_2+\partial_qh_2 \cdot \partial_p h_2 -p\cdot (\partial_{t,p}^2 h_2+ \partial^2_{z,p}h_2J\partial_z h_2)) +O(\sigma^3)
\end{align*} 
Note that the treatment of the term of order $\sigma$ has to be performed carefully in the case of $S_2(t^\flat,t^\flat+\sigma,\Phi_1^{t^\flat+\sigma,t^\flat}(z^\flat))$. 
We obtain 
\begin{align*}
S_2(t^\flat,t^\flat+\sigma,\Phi_1^{t^\flat+\sigma,t^\flat}(z^\flat)) =&  -\sigma (p\cdot \partial_p h_2 -h_2)
\\
&-\sigma^2(-\partial_t h_2-\partial_q h_2\cdot \partial_p  h_1+p\cdot (\partial^2_{t,p} h_2 +\partial^2_{z,p} h_2 J\partial_z h_2))\\
&-\frac{\sigma^2}{2} (\partial_t h_2+\partial_qh_2 \cdot \partial_p h_2 -p\cdot (\partial_{t,p}^2 h_2+ \partial^2_{z,p}h_2J\partial_z h_1))
+O(\sigma^3)\\
=& \,  (p\cdot \partial_p h_2 -h_2)\sigma 
\\
&+\frac{\sigma^2}{2} (\partial_t h_2+\partial_qh_2 \cdot \partial_p (2h_1-h_2) -p\cdot (\partial_{t,p}^2 h_2+ \partial^2_{z,p}h_2J\partial_z (2h_2-h_1)\\
&
+O(\sigma^3)
\end{align*}
As a consequence,
\begin{align*}
S_1(t^\flat+\sigma,t^\flat,z^\flat) + & S_2(t^\flat,t^\flat+\sigma,\Phi_1^{t^\flat+\sigma,t^\flat}(z^\flat))=\sigma \, p\cdot \partial_p (h_1-h_2) 
 + \frac{\sigma^2} 2(\partial_t (h_2-h_1)\\
 &-\partial_q h_2\cdot \partial_p(h_2-h_1)+\partial_p h_1 \cdot \partial_q (h_2-h_1) \\
 &+p\cdot (\partial_{t,p}^2(h_1-h_2)+\partial_{z,p}^2 (h_1-h_2)J\partial_z h_1+\partial_{z,p}^2 h_2 J\partial_z(h_1-h_2)))+O(\sigma^3).
\end{align*}
Combining with~\eqref{C5''}, we obtain 
$$\Lambda(\sigma)=  \frac{\sigma^2} 2(\partial_t (h_2-h_1)-\partial_q h_2\cdot \partial_p(h_2-h_1)+\partial_p h_1 \cdot \partial_q (h_2-h_1)) +O(\sigma^3),$$
whence~\eqref{sd}.

\end{proof}

\section{The operators $\mathcal T_{\mu,\alpha,\beta}$}\label{app:T}

We study here the operators   $\mathcal T_{\mu,\alpha,\beta}$ that are defined in~\eqref{transf2} for $(\mu,\alpha,\beta)\in \R^{2d+1}$.  An explicit computation gives the  following  useful connection  with the Fourier transform
 \beq\label{FT}
 {\mathcal F}{\mathcal T}_{\mu,\alpha, \beta}= {\mathcal T}_{\mu+\alpha\cdot\beta, \beta, -\alpha}{\mathcal F}.
\eeq
The next proposition sums up the main information that we will use about these operators. 

\begin{proposition}\label{prop:T}
Let $(\mu,\alpha,\beta)\in \R^{2d+1}$.  
\begin{enumerate}
\item  The operator   $\mathcal T_{\mu,\alpha,\beta}$ maps ${\mathcal S}(\R^d)$ into itself if and only if $\mu\not=0$.
 \item Moreover, if $\mu\not=0$,   $\mathcal T_{\mu,\alpha,\beta}$ is  a metaplectic transformation in the Hilbert space
 $L^2(\R^d)$  multiplied by a complex number: 
 \begin{equation}\label{eq:T}
 {\mathcal T}_{\mu,\alpha, \beta} = \sqrt{\frac{2\pi}{i\mu}}{\rm e}^{\frac{i}{4\mu}(\beta\cdot y -\alpha\cdot D_y)^2}.
 \end{equation}
\item If $\mu\not=0$, $\Gamma\in{\mathfrak S}^+(d)$ and $A\in{\mathcal C}^\infty(\R^{2d})$ is a polynomial function then there exists
$\Gamma_{\mu,\alpha,\beta,\Gamma}\in{\mathfrak S}^ +(d)$ such that
$$
{\mathcal T}_{\mu,\alpha,\beta}  ({\rm op}^w_1(A) g^\Gamma ) = 
\sqrt{\frac{2\pi}{i\mu}} {\rm op}^w_1(A\circ \Phi_{\alpha,\beta} (-(4\mu)^{-1} )g^{\Gamma_{\mu,\alpha,\beta,\Gamma}}
$$
where  $\Phi_{\alpha,\beta}$ satisfies~\eqref{def:Phi} and 
\beq\label{Deq}
\Gamma_{\mu,\alpha,\beta,\Gamma} =    \Gamma  -\frac{(\beta-\Gamma\alpha)\otimes (\beta-\Gamma\alpha) }{2\mu-\alpha\cdot\beta+\alpha\cdot\Gamma\alpha}.
\eeq
\end{enumerate}
\end{proposition}

\begin{remark}
The matrix $\Gamma_{\mu,\alpha,\beta,\Gamma}$ is in $\mathfrak S^+(d)$ since $g^{\Gamma_{\mu,\alpha,\beta,\Gamma}}$ is proved to be Schwartz class. 
It is also important to notice that 
$2\mu-\alpha\cdot\beta+\alpha\cdot \Gamma\alpha$ is non zero because its imaginary part is non zero.  
\end{remark}

\begin{proof}
Point (1) is linked with Point (2) and comes from the formula~\eqref{transf1} and\eqref{transf2}. Indeed, when $\mu\not=0$, 
equation~\eqref{eq:T} is an application of relation~\eqref{transf2} and of functional calculus  on the self-adjoint operator $(\beta\cdot y -\alpha\cdot D_y)^2$  and the Fourier-transform formula  of complex Gaussian functions:
\beq\label{FG}
\int_{-\infty}^{+\infty}{\rm e}^{is^2\mu}{\rm e}^{is\tau}ds =\sqrt{\frac{2\pi}{i\mu}}{\rm e}^{\frac{\tau^2}{4i\mu}}, \;{\rm with}\;\arg(i\mu)\in]-\pi, \pi[.
\eeq
It remains to analyze the case where $\mu=0$. The computations are different whether $\alpha\cdot\beta=0$ or not. We assume $\alpha\not=0$ and we set
$$\hat \alpha=\frac\alpha{|\alpha|},\;\; y= (y\cdot \hat\alpha)\hat\alpha +y_\perp.$$
Similar formulas can be obtained when $\beta\not=0$ using~\eqref{FT}. 
Let us first assume $\alpha\cdot\beta=0$.
\begin{align*}
\mathcal T_{0,\alpha,\beta} & = \int \e^{is\beta \cdot y_\perp} \varphi( y\cdot \hat \alpha \hat \alpha -s\alpha +y_\perp) ds\\
&=|\alpha|^{-1} \int \e^{i|\alpha|^{-1} (y\cdot \hat \alpha-\sigma) (\beta\cdot y_\perp)} \varphi (\sigma \hat \alpha +y_\perp)d\sigma\\
&= |\alpha|^{-1} \e^{i |\alpha|^{-1} (y\cdot \hat \alpha) (\beta \cdot y_\perp)} \mathcal F_{\alpha} \varphi\left( \frac{\beta \cdot y_\perp}{|\alpha|} ,y_\perp\right)
\end{align*}
where $\varphi\in{\mathcal S}(\R^d)$, $y_\perp=y-{\hat\alpha\cdot y} \hat\alpha$ and $\mathcal F_\alpha$ is the partial Fourier transform in the direction~$\alpha$. \\
In the case where $\alpha\cdot\beta\not=0$, we write 
\begin{align*}
\mathcal T_{0,\alpha,\beta} & =
 (2\pi)^{-1} \int_{\R^2} \e^{-is^2 \frac{\alpha\cdot\beta}{2}+is\beta\cdot y+i\eta (y\cdot\alpha -s) }
\mathcal F_\alpha\varphi(\eta,y_\perp) d\eta ds\\
&= \sqrt{\frac{i}{\pi\beta\cdot  \alpha}} \int \e^{ i\frac {(\beta\cdot y -\eta)^2 } {2\alpha\cdot \beta} +i\eta y\cdot\alpha} \mathcal F_\alpha\varphi(\eta,y_\perp) d\eta \\
&=  \sqrt{\frac{i}{\pi\beta\cdot  \alpha}} 
 \e^{i \frac{(\beta\cdot y)^2} { 2\beta\cdot \alpha}} \int \e^{-i\eta \frac{\beta_\perp\cdot y_\perp} {\beta\cdot\alpha} }
e^{i \frac{\eta^2} {2\beta\cdot\alpha}} \mathcal F_\alpha\varphi(\eta,y_\perp) d\eta \\
& =  \sqrt{\frac{i}{\pi\beta\cdot  \alpha}}  \e^{i \frac{(\beta\cdot y)^2} { 2\beta\cdot \alpha}} 
\int \e^{-i\eta \frac{\beta_\perp\cdot y_\perp} {\beta\cdot\alpha} }
\mathcal F_\alpha\left(e^{i \frac{(D_{y}\cdot \hat \alpha)^2} {2\beta\cdot\alpha} } \varphi\right)(\eta,y_\perp) d\eta\\
&= \sqrt{\frac{4i\pi}{\beta\cdot  \alpha}}  \e^{i \frac{(\beta\cdot y)^2} { 2\beta\cdot \alpha}}  
e^{i \frac{(D_{y}\cdot\hat\alpha)^2} {2\beta\cdot\alpha}}  \varphi\left(-\frac{\beta_\perp\cdot y_\perp} {\beta\cdot\alpha}+ y_\perp\right)
\end{align*}
This concludes the proof of Points~(1) and~(2).

\medskip
 
Point (3) derives from  the formulation of~$\mathcal T_{\mu,\alpha,\beta}$ as a metaplectic transform. We use general  results concerning the action of a metaplectic transformation  on Gaussian~$g^\Gamma$ (for details see \cite{ CombescureRobert}, Chapter 3).  With the quadratic Hamiltonian
 $K(y,\eta) = (\beta\cdot y-\alpha\cdot\eta)^2$, one associates the   linear flow $\Phi_{\alpha,\beta}(t)=(\Phi_{ij}(t))_{1\leq i,j\leq 2}$ (in a $d\times d$ block form) given by~\eqref{def:Phi}. Besides, the Egorov theorem and the propagation of gaussian are both  exact: we have 
 $$
{\rm e}^{-it\hat K}({\rm op}_1^w(A)g^\Gamma) = {\rm op}_1^w(A\circ \Phi_{\alpha,\beta}(t)){\rm e}^{-it\hat K} g^\Gamma=
({\rm op}_1^w(A\circ \Phi_{\alpha,\beta}(t))g^{\Gamma_t}
$$
where the matrix $\Gamma_t\in \mathfrak S^+(d)$ is given by
$$
\Gamma_t = (\Phi_{21}(t) +\Phi_{22}(t)\Gamma)(\Phi_{11}(t)+\Phi_{12}(t)\Gamma)^{-1} 
,\;\; c_{\Gamma_t} ={\rm det}^{-1/2}(A(t) +B(t)\Gamma),$$
where 
We deduce  that if $\mu\not=0$,
$$\mathcal T_{\mu,\alpha,\beta} g^\Gamma=\sqrt{\frac{2\pi}{i\mu}} \e^{\frac{i}{4\mu} \hat K} g^\Gamma= \sqrt{\frac{2\pi}{i\mu}} g^{\Gamma_{-(4\mu)^{-1}}}.$$
This induces the existence of the matrix $\Gamma_{\mu,\alpha,\beta,\Gamma}\in\mathfrak S^+(d)$ of Point (2) of the Proposition. It remains to prove the formula~\eqref{Deq}.
We use that if $\varphi =g_\Gamma$, we have
  $$
  {\mathcal T}_{\mu,\alpha,\beta} g^\Gamma(y) = c_\Gamma \int_{-\infty}^{+\infty}{\rm e}^{is^2(\mu-\alpha\cdot\beta/2)}{\rm e }^{is\beta\cdot y}
  {\rm e}^{\frac{i}{2}(y-s\alpha)\cdot(\Gamma(y-s\alpha))}ds.
  $$
  Applying again \eqref{FG} we get, 
\[
   {\mathcal T}_{\mu,\alpha,\beta} g^\Gamma(y)=c_\Gamma \sqrt{\frac{\pi}{2\mu-\alpha\cdot\beta+\alpha\cdot\Gamma\alpha}}\;
   {\rm e}^{\frac{i}{2}\left(y\cdot\Gamma y -\frac{(y\cdot(\beta-\Gamma\alpha))^2}{2\mu-\alpha\cdot\beta+\alpha\cdot\Gamma\alpha}\right)},
\]
which gives~\eqref{Deq}.
\end{proof}

\end{document}